\documentclass[11pt, reqno]{amsart}
\usepackage{color}
\usepackage{amsmath}
\usepackage{latexsym,amssymb}
\usepackage[mathscr]{euscript}
\usepackage{stmaryrd}
\usepackage{bm}
\usepackage{enumerate}
\usepackage{setspace}
\newtheorem{thm}{Theorem}[section]
\newtheorem{proposition}{Proposition}[section]
\newtheorem{example}{\bf Example}[section]

\newtheorem{lemma}{\bf Lemma}[section]
\newtheorem{definition}{Definition}[section]
\numberwithin{equation}{section}
\newtheorem{assumption}{Assumption}[section]

\begin{document}
	
	\baselineskip=17pt
	
	\title[]
	{ Zero-Sum Games for Continuous-time Markov decision processes with risk-sensitive average cost criterion}
	
	\author[M. K. Ghosh]{Mrinal K. Ghosh}
\address{ Department of Mathematics\\
	Indian Institute of Science\\
	Bangalore-560012, India.}
\email{mkg@iisc.ac.in}

\author[S. Golui]{Subrata Golui}
\address{Department of Mathematics\\
	Indian Institute of Technology Guwahati\\
	Guwahati, Assam, India}
\email{golui@iitg.ac.in}

\author[C. Pal]{Chandan Pal}
\address{Department of Mathematics\\
	Indian Institute of Technology Guwahati\\
	Guwahati, Assam, India}
\email{cpal@iitg.ac.in}

\author[S. Pradhan]{Somnath Pradhan}
\address{Department of Mathematics and Statistics\\
	Queen's University\\
	 Kingston, Ontario-K7L 3N6, Canada}
\email{sp165@queensu.ca}

	
	\date{}
	
	\begin{abstract}
		\vspace{2mm}
		\noindent
		 We consider zero-sum stochastic games for continuous time Markov decision processes with risk-sensitive average cost criterion. Here the transition and cost rates may be unbounded. We prove the existence of the value of the game and a saddle-point equilibrium in the class of all stationary strategies under a Lyapunov stability condition. This is accomplished by  establishing the existence of a principal eigenpair for the corresponding Hamilton-Jacobi-Isaacs (HJI) equation. This in turn is established by using the nonlinear version of Krein-Rutman theorem. We then obtain a characterization of the saddle-point equilibrium in terms of the corresponding HJI equation. Finally, we use a controlled population system to illustrate  results.
		
		\vspace{2mm}
		
		
	\end{abstract}
	
	\maketitle
	
\section{INTRODUCTION}
	 Markov decision processes (MDPs) are widely used for modeling control problems that arise naturally in many real-life problems, for example in queueing models, epidemiology models, birth-death models etc, see \cite{BR1}, \cite{GH4}, \cite{PZ1}, \cite{P}. When there is more than one controller (or player) the stochastic control problem is referred to as stochastic game problem. Stochastic dynamic game  was first introduced in \cite{SH} and has been studied extensively in the literature due to its immense applications; see \cite{BG}, \cite{BR}, \cite{GKP},  \cite{GP1}, \cite{GH3}, \cite{WC2}, \cite{W2}, \cite{ZG} and the references therein. In this article we consider the risk-sensitive ergodic zero-sum game for continuous-time Markov decision processes (CTMDPs). In zero-sum game, one player is trying to minimize her/his cost and the other player is trying to maximize the same. In literature, the expected average cost criterion is a commonly used optimality criterion in the theory of CTMDPs and has been widely studied under the different sets of optimality conditions;  for control problems see, \cite{GH4}, \cite{Z1} and the references therein; for game problems see \cite{GH1}, \cite{HL}, \cite{WC4} and the references therein. In these papers  the decision-makers are risk-neutral.  However, risk preferences may vary from person to person in the real-world applications. In order to address this concern one of the approaches that is available in the literature is risk-sensitive criterion. In this criterion one investigates the expectation of an exponential of the random quantity. This takes into account the attitude of the controller with respect to risk.  The performance of a pair of strategies is measured by risk-sensitive average cost criterion, which in our present case is defined by (\ref{eq 2.4}), below.  The analysis of risk-sensitive control is technically more involved because of the exponential nature of the cost. The risk-sensitive average cost stochastic optimal control problems for CTMDPs were first considered in \cite{GS2} and have been studied extensively in the literature due to its applications in finance and large deviation theory. Recently, there has been an extensive work on risk-sensitive average cost criterion problems for CTMDPs; see, for example  \cite{BP}, \cite{GYH}, \cite{KP1}, \cite{KP2}, \cite{PP} and the references therein. The risk-sensitive stochastic zero-sum games for MDPs have been studied in [\cite{BG}, \cite{BR}, \cite{GKP}, \cite{GP1}, \cite{W2}] and [\cite{BG1}, \cite{K1}, \cite{WC5}] consider the nonzero-sum games for MDPs.
	 In [\cite{BG}, \cite{BR}], the authors study zero-sum risk-sensitive stochastic games for discrete time MDPs with bounded cost. Both of the papers considered first the discounted cost and then ergodic cost. In \cite{BR}, the authors extended the results of \cite{BG} to the general state space case. The zero-sum risk-sensitive average games have been studied in \cite{GKP} and  discounted risk-sensitive zero-sum games were studied in  \cite{PP1} for CTMDPs with bounded cost and transition rates. But this boundedness requirement restricts our domain of application, since in many real-life situations we see that the reward/cost and transition rates are unbounded as for example in queueing, telecommunication and population processes. In \cite{GP1} and \cite{W2}, the authors study finite horizon zero-sum risk-sensitive continuous-time stochastic games. In \cite{GP1}, unbounded costs and transition rates are considered while \cite{W2} considers unbounded transition but bounded cost. The  discounted risk-sensitive zero-sum game for CTMDPs was studied in \cite{GP2} with unbounded cost and transition rates.
	
Here we study  zero-sum ergodic risk-sensitive stochastic games for CTMDPs having the following features: (a) transition and cost rates may be unbounded (b) state space is countable (c) at any state of the system the space of admissible actions is compact (d) the strategies may be history dependent. To the best of our knowledge, this is the first work which deals with infinite horizon continuous-time zero-sum risk-sensitive stochastic games for ergodic criterion on countable state space for unbounded transition and cost rates.  Under a Lyapunov stability condition,  we prove the existence of a saddle-point equilibrium in the class of stationary strategies. Using Krein-Rutman theorem, we first prove that the corresponding HJI equation has a unique solution for any finite subset of the state space. Then using the Lyapunov stability condition, we establish the existence of a unique solution for the corresponding HJI equation on the whole state space. Also we give a complete characterization of saddle point equilibrium in terms of the corresponding HJI equation.
	
The rest of this article is arranged as follows. Section 2 gives the description of the problem and assumptions. We also show in this section that the required risk-sensitive optimality equation (HJI equation) has a solution. In Section 3, we completely characterize all possible saddle point equilibria in the class of stationary Markov strategies. In Section 4, we present an illustrative example.

	\section{The game model}
	In this section we introduce the continuous-time zero-sum stochastic game model described by the following elements
	\begin{equation}
	\{S, A,B, (A(i)\subset A, B(i)\subset B,i\in S),q( \cdot|i, a,b),c(i, a,b)
	\}, \label{eq 2.1}
	\end{equation}
where
	\begin{itemize}
		\item  $S$, called the state space, is the set of all nonnegative integers. 	
		\item   $A$ and $B$ are the action sets for players 1 and 2, respectively. The action spaces $A$ and $B$ are assumed to be Borel spaces with the Borel $\sigma$-algebras $\mathcal{B}(A)$ and $\mathcal{B}(B)$, respectively.
		\item  For each $i\in S$, $A(i)\in \mathcal{B}(A)$ and $B(i)\in \mathcal{B}(B)$ denote the sets of admissible actions for players 1 and 2 in state $i$, respectively. Let $K:=\{(i, a,b)|i\in S, a\in A(i), b\in B(i)\}$, which is a Borel subset of $S\times A\times B$.
	Throughout this paper, we assume that the admissible action spaces $A(i)(\subset A)$ and $B(i)(\subset B)$ are compact for each $i$.
		\item Given any $(i, a,b)\in K$, the transition rate $q(j | i, a,b)$ is a signed kernel on $S$ such that  $q(j | i, a,b)\geq 0 $ for all $j, i\in S$ with $j\neq i$. Moreover, we assume that $q(j | i, a,b)$ satisfies the following conservative and stable conditions:
			for any $i\in S,$
		\begin{align}
		&\sum_{j\in S}q(j|i,a,b)=0~\text{for all}~(a,b)\in A(i)\times B(i) ~~~\text{and}\nonumber\\
		&~q^{*}(i):=\sup_{(a,b)\in A(i)\times B(i)}q(i,a,b)<\infty,\label{eq 2.2}
		\end{align}
		where $q(i,a,b):=-q(i | i, a,b)\geq 0.$
		\item Finally, the measurable function $c:K \to \mathbb{R}_{+}$ denotes the cost rate (representing cost for player 2 and payoff for player 1).
	\end{itemize}
		The game evolves as follows. The players observe continuously the current state of the system. When the system is in state $i\in S$ at time $t\geq0$, the players independently choose actions $a_t\in A(i)$ and  $b_t\in B(i)$ according to some strategies, respectively. As a consequence of this, the following happens:
	\begin{itemize}
		\item player 2 pays an immediate cost at rate $c(i, a_t, b_t )$ to player 1;
		\item the system stays in state $i$ for a random time, with rate of leaving $i$ given by $q(i,a_t,b_t)$, and then jumps to a new state $j\neq i$ with the probability determined by $\dfrac{q(\cdot|i,a_t,b_t)}{q(i,a_t,b_t)}$ (see Proposition in [\cite{GH4}, p. 205] for details).
	\end{itemize}
	When the state of the system transits to a new state $j$, the above procedure is repeated.
	
	The goal of player 2 is to minimize his/her accumulated cost, whereas  player 1 tries to maximize the same with respect to some performance criterion $J(\cdot,\cdot, \cdot,\cdot)$, which in our present case is defined by (\ref{eq 2.4}), below. Such a model is relevant in worst-case scenarios, e.g., in financial applications when a risk-averse investor is trying to maximize his long-term portfolio gain against the market which, by default, is the minimizer in this case.
	
	To formalize what is described above, below we describe the construction of continuous time Markov decision processes (CTMDPs) under possibly admissible history-dependent strategies.
	To construct the underlying CTMDPs (as in [\cite{GP}, \cite{K}, \cite{PZ}]), we introduce some notations: let $S_\Delta:=S \cup \{\Delta\}$ (with some $\Delta \notin S$), $\Omega^0:=(S\times(0,\infty))^\infty$, $\Omega:=\Omega^0 \cup \{(i_0, \theta_1, i_1, \cdots ,\theta_k, i_k, \infty, \Delta, \infty,\Delta, \cdots)| i_l \in S, \theta_l\in (0,\infty),$ \text{for each} $0\leq l\leq k, k\geq 0\}$, and let $\mathscr{F}$ be the Borel $\sigma$-algebra on $\Omega$. Then we obtain the measurable space $(\Omega, \mathscr{F})$.
	For each $k\geq 0$, $ \omega:=(i_0, \theta_1, i_1, \cdots , \theta_k, i_k, \cdots)\in \Omega,$ define $T_0(\omega):=0$, $T_k(\omega):= T_{k-1}(\omega)+\theta_{k}$, $T_\infty(\omega):=\lim_{k\rightarrow\infty}T_k(\omega)$. Using $\{T_k\}$, we define the state process $\{\xi_t\}_{t\geq 0}$ as
	\begin{equation}\label{eq 2.3}
	\xi_t(\omega):=\sum_{k\geq 0}I_{\{T_k\leq t<T_{k+1}\}}i_k+ I_{\{t\geq T_\infty\}}\Delta, \text{ for } t\geq 0.	
	\end{equation}
	Here, $I_{E}$ denotes the indicator function of a set $E$, and we use the convention that $0+z=:z$ and $0\cdot  z=:0$ for all $z\in S_\Delta$. The process after $T_\infty$ is regarded to be absorbed in the state $\Delta$. Thus, let $q(\cdot | \Delta, a_\Delta,b_\Delta):\equiv 0$, $A_\Delta:=A\cup \{a_\Delta\}$, $B_\Delta:=B\cup \{b_\Delta\}$, $ A(\Delta):=\{a_\Delta\}$, $B(\Delta):=\{b_\Delta\}$, $c(\Delta, a,b):\equiv 0$ for all $(a,b)\in A_\Delta\times B_\Delta$, where $a_\Delta$, $b_\Delta$ are isolated points. Moreover, let $\mathscr{F}_t:=\sigma(\{T_k\leq s,\xi_{T_k}\in D\}:D\in \mathcal{B}(S),0\leq s\leq t, k\geq0)$ for all $t\geq 0$, $\mathscr{F}_{s-}=:\bigvee_{0\leq t<s}\mathscr{F}_t$, and $\mathscr{P}:=\sigma(\{A\times \{0\},A\in \mathscr{F}_0\} \cup \{ B\times (s,\infty),B\in \mathscr{F}_{s-}\})$ which denotes the $\sigma$-algebra of predictable sets on $\Omega\times [0,\infty)$ related to $\{\mathscr{F}_t\}_{t\geq 0}$.
	
	In order to define the risk sensitive cost criterion, we need to introduce the definition of strategy below.

\begin{definition}
	An admissible history-dependent strategy for player 1, denoted by $\pi^1$, is determined by a sequence $\{\pi _k^1,k\geq 0\}$ of stochastic kernel on $A$ such that
	\begin{align*}
	\pi^1(da | \omega,t)&=I_{\{t=0\}}(t)\pi _0^1(da|i_0, 0)+\sum_{k\geq 0}I_{\{T_k< t\leq T_{k+1}\}}\pi^1 _k(da|i_0, \theta_1, i_1, \dots , \theta_k, i_k, t-T_k)\notag\\
	&+I_{\{t\geq  T_\infty\}}\delta_{a_\Delta}(da),
	\end{align*}
	where $\pi _0^1(da|i_0, 0)$ is a stochastic kernel on $A$ given $S$ such that $\pi _0^1(A(i_0)|i_0, 0)=1$, $\pi^1 _k (k\geq 1)$ are stochastic kernels on $A$ given $(S\times (0,\infty))^{k+1}$ such that $\pi _k^1(A(i_k)|i_0,\theta_1,i_1,\cdots,\theta_k,i_k,t-T_k )=1$, and $\delta_{a_\Delta}(da)$ denotes the Dirac measure at the point $a_\Delta$.
\end{definition}
The set of all admissible history-dependent strategies for player 1 is denoted by $\Pi^1$.
A strategy $\pi^1\in \Pi^1$ for player 1, is called a Markov if $\pi^1(da | \omega,t)=\pi^1(da | \xi_{t-}(w),t)$ for every $w\in \Omega$ and $t\geq 0$, where $\xi_{t-}(w):=\lim_{s\uparrow t}\xi_s(w)$. A Markov stragegy $\pi^1$ is called a stationary Markov strategy if $\pi^1$ does not have explicit dependence on time. We denote by  $\Pi^{m}_1$ and $\Pi^s_1$ the family of all Markov strategies and stationary Markov strategies, respectively, for player 1. The  sets of all admissible history-dependent strategies $\Pi^2$, all Markov strategies $\Pi^m_2$ and all stationary strategies $\Pi^s_2$ for player 2 are defined similarly.

For any compact metric space $Y$, let $P(Y)$ denote the space of probability measures on $\mathcal{B}(Y)$ with Prohorov topology. Since for each $i\in S$, $ A(i)$ and $ B(i)$ are compact sets,  $ P(A(i))$ and $ P(B(i))$ are compact metric spaces. For each $i ,j\in  S$, $\mu\in P(A(i))$ and $\nu\in P(B(i))$, the associated cost and transition rates are defined, respectively, as follows:
$$c(i,\mu,\nu):=\int_{B(i)}\int_{A(i)}c(i,a,b)\mu(da)\nu(db),$$
$$q(j|i,\mu,\nu):=\int_{B(i)}\int_{A(i)}q(j|i,a,b)\mu(da)\nu(db).$$
Note that $\pi^1 \in {\Pi_{1}^s}$ can be identified with a map $\pi^1: S \to {P}(A)$ such that  $ \pi^1(\cdot|j) \in {P}(A(j))$ for each $j \in S$. Thus, we have $  \Pi_{1}^s=\displaystyle \Pi_{i\in S} {P}(A(i))$ and $\Pi_{2}^s=\displaystyle  \Pi_{i\in S} {P}(B(i))$. Therefore by Tychonoff theorem, the sets ${\Pi_{1}^s}$ and ${\Pi_{2}^s}$ are compact metric spaces.
Also, note that under Assumption \ref{assm 2.1} (given below) for any initial state $i\in S$ and any pair of strategies $(\pi^1,\pi^2)\in \Pi^1\times \Pi^2$, Theorem 4.27 in \cite{KR1} yields the existence of a unique probability measure denoted by $P^{\pi^1,\pi^2}_i$ on $(\Omega,\mathscr{F})$.  Let $E^{\pi^1,\pi^2}_i$ be the expectation operator with respect to  $P^{\pi^1,\pi^2}_i$. Also, from [\cite{GH4}, pp.13-15], we know that $\{\xi_t\}_{t\geq 0}$ is a Markov process under any $(\pi^1,\pi^2)\in \Pi^m_1\times\Pi^m_2$ (in fact, strong Markov).
 Now we give the definition of the risk-sensitive average cost criterion for zero-sum continuous-time games. Since the risk-sensitive parameter remains fixed throughout  we assume
  without any loss of generality that the risk-sensitivity coefficient $\theta=1$. For each $i\in S$ and any $(\pi^1,\pi^2) \in \Pi^1\times \Pi^2$, the risk-sensitive ergodic cost criterion is given by
\begin{equation}
J(i,c,\pi^1,\pi^2):= \limsup_{T\rightarrow \infty}\frac{1}{ T}\ln E^{\pi^1,\pi^2}_i\biggl[e^{\int_{0}^{T}\int_{B}\int_{A} c(\xi_t,a,b)\pi^1(da|\omega,t)\pi^2(db|\omega,t)dt}\biggr].\label{eq 2.4}
\end{equation}
Player 1 tries to maximize the above over his/her admissible strategies whereas player 2 tries to minimize the same.
Now we define the lower/upper value of the game. The functions on S defined by $L(i):=\displaystyle \sup_{\pi^1\in \Pi^1}\inf_{\pi^2\in \Pi^2}J(i,c,\pi^1,\pi^2)$ and $U(i):=\displaystyle \inf_{\pi^2\in \Pi^2}\sup_{\pi^1\in \Pi^1}J(i,c,\pi^1,\pi^2)$ are called, respectively, the lower value and the upper value of the game. It is easy to see that
$$L(i)\leq U(i)~ \text{for all}~i\in S.$$
\begin{definition}
	If $L(i)=U(i)$ for all $i\in S$, then the common function is called the value of the game and is denoted by $J^{*}(i)$.
\end{definition}

\begin{definition}
	Suppose that the game admits a value $J^{*}$. Then a strategy $\pi^{*1}$ in $\Pi^1$ is said to be optimal for player 1 if
	$$\inf_{\pi^2\in \Pi^2}J(i,c,\pi^{*1},\pi^2)=J^{*}(i)~ \text{for all}~ i\in S.$$
	Similarly, $\pi^{*2}\in \Pi^2$ is optimal for player 2 if
	$$\sup_{\pi^1\in \Pi^1}J(i,c,\pi^1,\pi^{*2})=J^{*}(i) ~ \text{for all}~i\in S.$$
	If $\pi^{*k}\in \Pi^k$ is optimal for player k (k=1,2), then $(\pi^{*1},\pi^{*2})$ is called a pair of optimal strategies and also called a saddle-point equilibrium.
\end{definition}
Next we list the commonly used notations below:
\begin{itemize}
	\item For any finite set $\mathcal{D}\subset S$, we define $\mathcal{B}_{\mathcal{D}} = \{f:S\to\mathbb{R}\mid \,\, f(i) = 0\,\,\, \forall \,\, i\in \mathcal{D}^c\}$\,.
	\item $\mathcal{B}_{\mathcal{D}}^{+}\subset \mathcal{B}_{\mathcal{D}}$ denotes the cone of all nonnegative functions vanishing outside $\mathscr{D}.$
	\item  Given any real-valued function $\mathcal{V}\geq 1$ on $S$, we define a Banach space $(L^\infty_{\mathcal{V}},\|\cdot\|^\infty_\mathcal{V})$ of $\mathcal{V}$-weighted  functions by
	$$L^\infty_\mathcal{V}=\biggl\{u: S\rightarrow\mathbb{R}\mid \|u\|^\infty_\mathcal{V}:=\sup_{i\in  S}\frac{|u(i)|}{\mathcal{V}(i)}< \infty\biggr\}.$$
\item $\|c\|_\infty:=\displaystyle \sup_{(i,a,b)\in K}c(i,a,b)$.
\item For any function $f\in \mathcal{B}_{\mathcal{D}}$, $\|f\|_\mathcal{D}=\max\{|f(i)|:i\in \mathcal{D}\}$.
\item For any finite set $\mathscr{B}\subset S$, $\tilde{\tau}(\mathscr{B}):=\inf\{t>0:\xi_t\in \mathscr{B}\}$.
\end{itemize}
Our main goal is to establish the existence of a saddle-point equilibrium among the class of admissible history-dependent strategies.
 To this end, following \cite{BG} and \cite{BP}, we investigate the HJI equation given by
\begin{align}
\rho \psi(i)&=\sup_{\mu\in P(A(i))}\inf_{\nu\in P(B(i))}\bigg[\sum_{j\in S}\psi(j)q(j|i,\mu,\nu)+c(i,\mu,\nu)\psi(i)\bigg]\nonumber\\
&=\inf_{\nu\in P(B(i))}\sup_{\mu\in P(A(i))}\bigg[\sum_{j\in S}\psi(j)q(j|i,\mu,\nu)+c(i,\mu,\nu)\psi(i)\bigg].
	\end{align}
Here $\rho$ is a scalar and $\psi$ is an appropriate function. The above is clearly an eigenvalue problem related to a nonlinear operator on an appropriate space.
By a nonlinear version of Krein-Rutman theorem, we first show that Dirichlet eigenvalue problem associated with the above equation admits a solution in the space of bounded functions. Then by using a suitable limiting argument we show that the above HJI equation admits a principal eigenpair in an appropriate space. Finally exploiting the HJI equation, we completely characterize all possible saddle-point equilibria in the space of stationary Markov strategies. This is a brief outline of our procedure of establishing a saddle point equilibrium.
 The details now follow.

Since the transition rates (i.e., $q(j |i, a,b)$ ) may be unbounded, to avoid the explosion of the state process $\{\xi_t, t\geq 0\}$,  the following assumption is imposed on the transition rates, which had been widely used in CTMDPs; see, for instance, [\cite{GS1}, \cite{GL}, \cite{GP}] and references therein.
\begin{assumption}\label{assm 2.1}
	There exist  real-valued function $\tilde{V}\geq 1$ on $S$, constants $b_0\neq 0$ and $b_1\geq 0$, and $b_2>0$ such that :
	\begin{enumerate}
		\item [(i)] $\sum_{j\in S}\tilde{V}(j)q(j | i, a,b)\leq b_0\tilde{V}(i)+b_1$ for all $(i, a,b)\in K$;
			\item [(ii)] $q^{*}(i)\leq b_2 \tilde{V}(i)$ for all $i\in S$, where $q^{*}(i)$ is as in (\ref{eq 2.2}).
	\end{enumerate}
\end{assumption}
Throughout the rest of this article we are going to assume that Assumption \ref{assm 2.1} holds. Note that if $\sup_{i \in S}q^{*}(i)<\infty$ then Assumption \ref{assm 2.1} holds trivially. In this case we can choose $\tilde{V}$ to be a suitable constant.

Since we are allowing our transition and cost rates to be unbounded, to guarantee the finiteness of $J(i,c,\pi^1,\pi^2)$, we need the following Assumption.
\begin{assumption}\label{assm 2.2}
	We assume that the CTMDP $\{\xi_t\}_{t\geq 0}$ is irreducible under every  pair of stationary Markov strategies $(\pi^1,\pi^2)\in \Pi^s_1 \times \Pi^s_2$. Assume  that the cost function $c$ is bounded below. Thus without loss of generality we assume that $c \geq 0$.
	Furthermore, suppose there exist a constant $C>0$, a finite set $\hat{\mathscr{K}}$ and a Lyapunov function $V : S \to [1,\infty)$ such that one of the following hold.
	\begin{itemize}
		\item[(a)]\textbf{When the running cost is bounded:} For some positive constant $\hat{\gamma} > \|c\|_{\infty}$, we have following blanket stability condition
		\begin{align}
	\sup_{(a,b)\in A(i)\times B(i)}\sum_{j\in S}V(j)q(j|i,a,b)\leq C I_{\hat{\mathscr{K}}}(i)-\hat{\gamma} V(i) ~\forall i\in S.\label{eq 2.5}
		\end{align}
	
		\item[(b)]\textbf{When the running cost is unbounded:} For some norm-like function $\hat{\ell} :S\rightarrow\mathbb{R}_{+}$, the function $\hat{\ell}(\cdot)-\displaystyle \max_{(a,b)\in A(\cdot)\times B(\cdot)}c(\cdot,a,b)\;$ is norm-like and we have the following blanket stability condition
		\begin{align}
	\sup_{(a,b)\in A(i)\times B(i)}\sum_{j\in S}V(j)q(j|i,a,b)\leq C I_{\hat{\mathscr{K}}}(i)-\hat{\ell}(i)V(i) ~\forall i\in S.\label{eq 2.6}
		\end{align}
	\end{itemize}
\end{assumption}
\noindent We wish to establish the  existence of a saddle-point equilibrium in the class of all stationary strategies. In view of this we also need the following assumptions. Let $i_0 \in S$ be a fixed point (a reference state).
\begin{assumption}\label{assm 2.3}
	\begin{enumerate}
		\item [(i)] For any fixed $i,  j\in S$ the functions  $q(j|i,a,b)$ and ${c}(i, a,b)$ are continuous in $(a,b)\in A(i)\times B(i)$\,.
		
		\item [(ii)] The sum $\displaystyle \sum_{j\in S}V(j)q(j|i,a,b)$ is continuous in $(a,b)\in A(i)\times B(i)$ for any given $i\in S$, where $V$ is as Assumption \ref{assm 2.2}.
		\item [(iii)] There exists $i_0\in S$ such that any state can be reached from $i_0$, i.e., $q(j|i_0,a,b)>0$ for all $j\neq i_0$ and $(a,b)\in A(i_0)\times B(i_0)$.
	\end{enumerate}	
\end{assumption}
We first construct an increasing sequence of finite subsets $\hat{\mathscr{D}}_n\subset S$ such that $\displaystyle \cup_{i=0}^{\infty}\hat{\mathscr{D}}_n=S$ and $i_0\in \hat{\mathscr{D}}_n$ for all $n\in \mathbb{N}$. Define $\tau_n:=\tau(\hat{\mathscr{D}}_n):=\inf\{t\geq 0:\xi_t\notin \hat{\mathscr{D}}_n\}$, first exit time from $\hat{\mathscr{D}}_n$.
\begin{proposition}\label{prop 2.1}
	Suppose Assumption \ref{assm 2.3} holds. Let $\tilde{c} : K \to \mathbb{R}$  be a function continuous in $(a,b)\in A(i)\times B(i)$ for each fixed $i\in S$.
Suppose the cost function $\tilde{c}$ satisfies the relation $\tilde{c} < - \delta$ in $\hat{\mathscr{D}}_n$ for some $\delta > 0$ and $n\in\mathbb{N}$\,. Then for any $g\in \mathcal{B}_{\hat{\mathscr{D}}_n}$ there exists a unique $\varphi\in \mathcal{B}_{\hat{\mathscr{D}}_n}$ satisfying the following nonlinear equation
\begin{align}
-g(i)&=\sup_{\mu\in P(A(i))}\inf_{\nu\in P(B(i))}\biggl[\sum_{j\in S}\varphi(j)q(j|i,\mu,\nu) + \tilde{c}(i,\mu,\nu)\varphi(i)\biggr]\nonumber\\
&=\inf_{\nu\in P(B(i))}\sup_{\mu\in P(A(i))}\biggl[\sum_{j\in S}\varphi(j)q(j|i,\mu,\nu) + \tilde{c}(i,\mu,\nu)\varphi(i)\biggr]~\forall i\in \hat{\mathscr{D}}_n, \label{eq 2.7}
\end{align}
with $\varphi(i)=0$ for all $i\in \hat{\mathscr{D}}_n^c$\,.
Moreover the unique solution of the above equation satisfies
\begin{align}
\varphi(i)&=\sup_{\pi^1\in \Pi^1}\inf_{\pi^2\in \Pi^2}E^{\pi^1,\pi^2}_i\biggl[\int_{0}^{\tau_n}e^{\int_{0}^{t} \tilde{c}(\xi_s,\pi^1(s),\pi^2(s))ds}g(\xi_t)dt\biggr]\nonumber\\
&=\inf_{\pi^2\in \Pi^2}\sup_{\pi^1\in \Pi^1}E^{\pi^1,\pi^2}_i\biggl[\int_{0}^{\tau_n}e^{\int_{0}^{t} \tilde{c}(\xi_s,\pi^1(s),\pi^2(s))ds}g(\xi_t)dt\biggr]~\forall i\in S,\label{eq 2.8}
\end{align}
where as before $\tau_n=\inf\{t\geq 0:\xi_t\notin \hat{\mathscr{D}}_n\}$.
\begin{proof}
	Let $(y_i)_{i\in \hat{\mathscr{D}}_n}$ be a sequence in $\mathbb{R}$. Fix $i\in \hat{\mathscr{D}}_n$. Let $F:\mathbb{R}\rightarrow \mathbb{R}$ be defined by
	\begin{align}
	x\rightarrow F(x)=\sup_{\mu\in P(A(i))}\inf_{\nu\in P(B(i))}\biggl[\sum_{i\neq j\in \hat{\mathscr{D}}_n}y_jq(j|i,\mu,\nu)+\biggl(q(i|i,\mu,\nu)+ \tilde{c}(i,\mu,\nu)\biggr)x\biggr], i\in \hat{\mathscr{D}}_n.\label{eq 2.9}
	\end{align}
	Suppose $x_2>x_1$. Let $\varepsilon>0$. Then there exists  $\pi^1_\varepsilon \in \Pi^s_1$ for which the following holds
	\interdisplaylinepenalty=0
	\begin{align*}
	F(x_1)-F(x_2)
	=&\sup_{\mu\in P(A(i))}\inf_{\nu\in P(B(i))}\biggl[\sum_{i\neq j\in \hat{\mathscr{D}}_n}y_jq(j|i,\mu,\nu)+\biggl(q(i|i,\mu,\nu)+ \tilde{c}(i,\mu,\nu)\biggr)x_1\biggr]\\
	&-\sup_{\mu\in P(A(i))}\inf_{\nu\in P(B(i))}\biggl[\sum_{i\neq j\in \hat{\mathscr{D}}_n}y_jq(j|i,\mu,\nu)+\biggl(q(i|i,\mu,\nu)+ \tilde{c}(i,\mu,\nu)\biggr)x_2\biggr]\\
	\geq & \inf_{\nu\in P(B(i))}\biggl[\sum_{i\neq j\in \hat{\mathscr{D}}_n}y_jq(j|i,\pi^1_\varepsilon(i),\nu)+\biggl(q(i|i,\pi^1_\varepsilon(i),\nu)+ \tilde{c}(i,\pi^1_\varepsilon(i),\nu)\biggr)x_1\biggr]\\
	&-\inf_{\nu\in P(B(i))}\biggl[\sum_{i\neq j\in \hat{\mathscr{D}}_n}y_jq(j|i,\pi^1_\varepsilon(i),\nu)+\biggl(q(i|i,\pi^1_\varepsilon(i),\nu)+ \tilde{c}(i,\pi^1_\varepsilon(i),\nu)\biggr)x_2+\varepsilon\biggr]\\
	\geq & \inf_{\nu\in P(B(i))}\biggl[\biggl(q(i|i,\pi^1_\varepsilon(i),\nu)+ \tilde{c}(i,\pi^1_\varepsilon(i),\nu)\biggr)(x_1-x_2)\biggr]-\varepsilon\\
	\geq & \inf_{\nu\in P(B(i))}\biggl[- \tilde{c}(i,\pi^1_\varepsilon(i),\nu)(x_2-x_1)\biggr]-\varepsilon\\
	> &\,\, \delta(x_2-x_1)-\varepsilon.
	\end{align*}
	Since $\varepsilon > 0$ is arbitrary we get $F(x_1)>F(x_2)$. Also, we see that $\lim_{x\rightarrow +\infty}F(x)=-\infty$ and $\lim_{x\rightarrow -\infty}F(x)=+\infty$. Since $F$ is continuous in $x$, for every $y\in \mathbb{R}$, there exists a unique $x$ satisfying $F(x)=y$. Now using the definition of $F$, for fixed $g\in \mathcal{B}_{\hat{\mathscr{D}}_n}$, we can define a map $\hat{T}:\mathcal{B}_{\hat{\mathscr{D}}_n}\rightarrow\mathcal{B}_{\hat{\mathscr{D}}_n}$ satisfying
	\begin{align}
	\sup_{\mu\in P(A(i))}\inf_{\nu\in P(B(i))}\biggl[\sum_{i\neq j\in \hat{\mathscr{D}}_n}\tilde{\phi}(j)q(j|i,\mu,\nu)+\biggl(q(i|i,\mu,\nu)+ \tilde{c}(i,\mu,\nu)\biggr)(\hat{T}\tilde{\phi}(i))\biggr]=-g(i),~i\in\hat{\mathscr{D}}_n.\label{eq 2.10}
	\end{align}
Let $\tilde{\phi}_1, \tilde{\phi}_2\in \mathcal{B}_{\hat{\mathscr{D}}_n}$. Also, let $\tilde{\pi}^1$ be an outer maximizing selector of $$\sup_{\mu\in P(A(i))}\inf_{\nu\in P(B(i))}\biggl[\sum_{i\neq j\in \hat{\mathscr{D}}_n}\tilde{\phi}_2(j)q(j|i,\mu,\nu)+\biggl(q(i|i,\mu,\nu)+ \tilde{c}(i,\mu,\nu)\biggr)\hat{T}\tilde{\phi}_2(i)\biggr]\,.$$
Assumption \ref{assm 2.3}, ensures the existence of such a selector.
It then follows that
\interdisplaylinepenalty=0
\begin{align*}
0= &\sup_{\mu\in P(A(i))}\inf_{\nu\in P(B(i))}\biggl[\sum_{i\neq j\in \hat{\mathscr{D}}_n}\tilde{\phi}_1(j)q(j|i,\mu,\nu)+\biggl(q(i|i,\mu,\nu)+ \tilde{c}(i,\mu,\nu)\biggr)\hat{T}\tilde{\phi}_1(i)\biggr]\nonumber\\
&-\sup_{\mu\in P(A(i))}\inf_{\nu\in P(B(i))}\biggl[\sum_{i\neq j\in \hat{\mathscr{D}}_n}\tilde{\phi}_2(j)q(j|i,\mu,\nu)+\biggl(q(i|i,\mu,\nu)+ \tilde{c}(i,\mu,\nu)\biggr)\hat{T}\tilde{\phi}_2(i)\biggr]\nonumber\\
\geq &\inf_{\nu\in P(B(i))}\biggl[\sum_{i\neq j\in \hat{\mathscr{D}}_n}\tilde{\phi}_1(j)q(j|i,\tilde{\pi}^1(i),\nu)+\biggl(q(i|i,\tilde{\pi}^1(i),\nu)+ \tilde{c}(i,\tilde{\pi}^1(i),\nu)\biggr)\hat{T}\tilde{\phi}_1(i)\biggr]\nonumber\\
&-\inf_{\nu\in P(B(i))}\biggl[\sum_{i\neq j\in \hat{\mathscr{D}}_n}\tilde{\phi}_2(j)q(j|i,\tilde{\pi}^1(i),\nu)+\biggl(q(i|i,\tilde{\pi}^1(i),\nu)+ \tilde{c}(i,\tilde{\pi}^1(i),\nu)\biggr)\hat{T}\tilde{\phi}_2(i)\biggr]\nonumber\\
\geq & \inf_{\nu\in P(B(i))}\biggl[\sum_{i\neq j\in \hat{\mathscr{D}}_n}(\tilde{\phi}_1(j)-\tilde{\phi}_2(j))q(j|i,\tilde{\pi}^1(i),\nu)+\biggl(q(i|i,\tilde{\pi}^1(i),\nu)+ \tilde{c}(i,\tilde{\pi}^1(i),\nu)\biggr)(\hat{T}\tilde{\phi}_1(i)-\hat{T}\tilde{\phi}_2(i))\biggr]\,.\nonumber
\end{align*}	
Now let the infimum of the RHS (of the above) attain at $\pi^{*2}$. Then
$$\|\tilde{\phi}_1-\tilde{\phi}_2\|_{\hat{\mathscr{D}}_n}q(i|i,\tilde{\pi}^1(i),\pi^{*2}(i))+\biggl(q(i|i,\tilde{\pi}^1(i),\pi^{*2}(i))+ \tilde{c}(i,\tilde{\pi}^1(i),\pi^{*2}(i))\biggr)(\hat{T}\tilde{\phi}_1(i)-\hat{T}\tilde{\phi}_2(i))\leq 0.$$
Hence, we deduce that
\interdisplaylinepenalty=0
\begin{align*}
(\hat{T}\tilde{\phi}_2(i)-\hat{T}\tilde{\phi}_1(i))\leq \sup_{\mu\in P(A(i))}\sup_{\nu\in P(B(i))}\frac{-q(i|i,\mu,\nu)}{-q(i|i,\mu,\nu)- \tilde{c}(i,\mu,\nu)}\|\tilde{\phi}_1-\tilde{\phi}_2\|_{\hat{\mathscr{D}}_n}\,.
\end{align*}
Now in the above calculation, interchanging $\tilde{\phi}_1,\tilde{\phi}_2$, it follows that
$$\|\hat{T}\tilde{\phi}_1-\hat{T}\tilde{\phi}_2\|_{\hat{\mathscr{D}}_n}\leq \alpha_1 \|\tilde{\phi}_1-\tilde{\phi}_2\|_{\hat{\mathscr{D}}_n}\,,$$ where $\alpha_1$ is a positive constant less than $1$\,. This implies that $\hat{T}$ is a contraction map. Thus, by Banach fixed point theorem, there exists a unique $\varphi\in \mathcal{B}_{\hat{\mathscr{D}}_n}$ such that $\hat{T}(\varphi)=\varphi$.
Now by Fan's minimax theorem, see [\cite{Fan}, Theorem 3], we have
\begin{align*}
&\sup_{\mu\in P(A(i))}\inf_{\nu\in P(B(i))}\biggl[\sum_{j\in S}\varphi(j)q(j|i,\mu,\nu)+ \tilde{c}(i,\mu,\nu)\varphi(i)\biggr]\nonumber\\
&=\inf_{\nu\in P(B(i))}\sup_{\mu\in P(A(i))}\biggl[\sum_{j\in S}\varphi(j)q(j|i,\mu,\nu)+ \tilde{c}(i,\mu,\nu)\varphi(i)\biggr].
\end{align*}
This proves that (\ref{eq 2.7}) admits a unique solution. Now by using Dynkin formula as in [\cite{GH4}, Appendix C.3], for any $(\pi^1,\pi^2)\in \Pi^1\times \Pi^2$ and $T>0$, we get
\interdisplaylinepenalty=0
\begin{align}
&E^{\pi^1,\pi^2}_i\biggl[e^{\int_{0}^{T\wedge\tau_n} \tilde{c}(\xi_s,\pi^1(s),\pi^2(s))ds}\varphi(\xi_{T\wedge\tau_n})\biggr]-\varphi(i)\nonumber\\
&=E^{\pi^1,\pi^2}_i\biggl[\int_{0}^{T\wedge\tau_n}e^{\int_{0}^{t} \tilde{c}(\xi_s,\pi^1(s),\pi^2(s))ds}\bigg(\tilde{c}(\xi_t,\pi^1(t),\pi^2(t))\varphi(\xi_t)+\sum_{j\in S}\varphi(j)q(j|\xi_t,\pi^1(t),\pi^2(t))\bigg)dt\biggr].\label{eq 2.11}
\end{align}
 Using the compactness of $A(i), B(i)$ and the continuity of $\tilde{c}, q$,  there exists a pair of selectors $(\pi^{*1},\pi^{*2})\in \Pi^s_1\times \Pi^s_2$ (i.e., a mini-max selector) satisfying
\interdisplaylinepenalty=0
\begin{align}
-g(i)&=\inf_{\nu\in P(B(i))}\biggl[\sum_{j\in S}\varphi(j)q(j|i,\pi^{*1}(i),\nu)+ \tilde{c}(i,\pi^{*1}(i),\nu)\varphi(i)\biggr]\nonumber\\
&=\sup_{\mu\in P(A(i))}\biggl[\sum_{j\in S}\varphi(j)q(j|i,\mu,\pi^{*2}(i))+ \tilde{c}(i,\mu,\pi^{*2}(i))\varphi(i)\biggr].\label{eq 2.12}
\end{align}
Then, using (\ref{eq 2.11}) and (\ref{eq 2.12}), we obtain
\begin{align*}
E^{\pi^{*1},\pi^2}_i\biggl[\int_{0}^{T\wedge\tau_n}e^{\int_{0}^{t} \tilde{c}(\xi_s,\pi^{*1}(\xi_s),\pi^2(s))ds}g(\xi_t)dt\biggr]\geq -E^{\pi^{*1},\pi^2}_i\biggl[e^{\int_{0}^{T\wedge\tau_n} \tilde{c}(\xi_s,\pi^{*1}(\xi_s),\pi^2(s))ds}\varphi(\xi_{T\wedge\tau_n})\biggr]+\varphi(i).
\end{align*}
Using the dominated convergence theorem, taking $T\rightarrow\infty$ in the above equation, we get
\interdisplaylinepenalty=0
\begin{align*}
E^{\pi^{*1},\pi^2}_i\biggl[\int_{0}^{\tau_n}e^{\int_{0}^{t} \tilde{c}(\xi_s,\pi^{*1}(\xi_s),\pi^2(s))ds}g(\xi_t)dt\biggr]&\geq -E^{\pi^{*1},\pi^2}_i\biggl[e^{\int_{0}^{\tau_n} \tilde{c}(\xi_s,\pi^{*1}(\xi_s),\pi^2(s))ds}\varphi(\xi_{\tau_n})\biggr]+\varphi(i)\\
&=\varphi(i).
\end{align*}
Hence
\begin{align*}
\varphi(i)&\leq E^{\pi^{*1},\pi^2}_i\biggl[\int_{0}^{\tau_n}e^{\int_{0}^{t} \tilde{c}(\xi_s,\pi^{*1}(\xi_s),\pi^2(s))ds}g(\xi_t)dt\biggr].
\end{align*}
Since $\pi^2\in \Pi^2$ is arbitrary,
\begin{align}
\varphi(i)\leq \inf_{\pi^2\in \Pi^2} E^{\pi^{*1},\pi^2}_i\biggl[\int_{0}^{\tau_n}e^{\int_{0}^{t} \tilde{c}(\xi_s,\pi^{*1}(\xi_s),\pi^2(s))ds}g(\xi_t)dt\biggr].\label{eq 2.13}
\end{align}
Similarly, using (\ref{eq 2.11}), (\ref{eq 2.12}), and Fatou's Lemma, we get
\begin{align}
\varphi(i)&\geq\sup_{\pi^1\in \Pi^1} E^{\pi^{1},\pi^{*2}}_i\biggl[\int_{0}^{\tau_n}e^{\int_{0}^{t} \tilde{c}(\xi_s,\pi^{1}(s),\pi^{*2}(\xi_s))ds}g(\xi_t)dt\biggr].\label{eq 2.14}
\end{align}
Using (\ref{eq 2.13}) and (\ref{eq 2.14}), we obtain
\begin{align*}
\varphi(i)&=\inf_{\pi^2\in \Pi^2}\sup_{\pi^1\in \Pi^1} E^{\pi^{1},\pi^2}_i\biggl[\int_{0}^{\tau_n}e^{\int_{0}^{t} \tilde{c}(\xi_s,\pi^{1}(s),\pi^2(s))ds}g(\xi_t)dt\biggr]\\
&=\sup_{\pi^1\in \Pi^1}\inf_{\pi^2\in \Pi^2} E^{\pi^{1},\pi^2}_i\biggl[\int_{0}^{\tau_n}e^{\int_{0}^{t} \tilde{c}(\xi_s,\pi^{1}(s),\pi^2(s))ds}g(\xi_t)dt\biggr]~i\in S.
\end{align*}
This completes the proof.
\end{proof}
\end{proposition}
 We now recall a version of the nonlinear Krein-Rutman theorem from [\cite{A1}, Section 3.1].
 Let $\hat{\mathscr{X}}$ be an ordered Banach space. In what follows $\succeq$ denotes a partial ordering in $\hat{\mathscr{X}}$ with respect to a positive cone $\hat{\mathscr{C}}$ ($\subset \hat{\mathscr{X}}$), that is $x\succeq y\Leftrightarrow x-y\in \hat{\mathscr{C}}$. Also, recall that if a map $\tilde{T}:\hat{\mathscr{X}}\rightarrow\hat{\mathscr{X}}$ is continuous and compact, it is called completely continuous.
\begin{thm}\label{theo 2.1}
 Let $\hat{\mathscr{X}}$ be as above and $\hat{\mathscr{C}}\subset \hat{\mathscr{X}}$  a nonempty closed cone that satisfies $\hat{\mathscr{C}}-\hat{\mathscr{C}}=\hat{\mathscr{X}}$. Let $\tilde{T}:\hat{\mathscr{X}}\rightarrow\hat{\mathscr{X}}$ be an order-preserving, completely continuous, 1-homogeneous map with the property that if for some nonzero $\zeta\in \hat{\mathscr{C}}$ and $N>0$, we have  $N\tilde{T}(\zeta)\succeq \zeta$.  Then there exist a nontrivial $f\in \hat{\mathscr{C}}$ and $\tilde{\lambda}>0$ satisfying $\tilde{T}f=\tilde{\lambda}f$.
\end{thm}
\begin{lemma}\label{lemm 2.4}
	Suppose Assumption \ref{assm 2.2} holds. Consider a finite subset $\mathscr{B}$ of $S$ such that $\hat{\mathscr{K}}\subset \mathscr{B}$. Let $\tilde{\tau}(\mathscr{B})=\inf\{t>0:\xi_t\in \mathscr{B}\}$. Then for any pair of strategies $(\pi^1,\pi^2)\in \Pi^1\times \Pi^2$,  the following results hold.
	\begin{enumerate}
		\item [(i)] When Assumption \ref{assm 2.2} (a) holds:
		\begin{align}
		E^{\pi^1,\pi^2}_i\bigg[e^{\hat{\gamma} \tilde{\tau}(\mathscr{B})}V(\xi_{\tilde{\tau}(\mathscr{B})})\bigg]\leq V(i)~\forall~i\in \mathscr{B}^c.\label{eq 2.21}
		\end{align}
		\item [(ii)] When Assumption \ref{assm 2.2} (b) holds:
		\begin{align}
		E^{\pi^1,\pi^2}_i\bigg[e^{\int_{0}^{\tilde{\tau}(\mathscr{B})}\hat{\ell}(\xi_s) ds}V(\xi_{\tilde{\tau}(\mathscr{B})})\bigg]\leq V(i)~\forall~i\in \mathscr{B}^c.\label{eq 2.22}
		\end{align}
		\end{enumerate}
		\begin{proof}
			It is easy to see that the proof of (i) is analogous to that the proof of (ii) when we replace $\hat{\ell}$ with $\hat{\gamma}$. So, we  prove only part (ii). Suppose Assumption \ref{assm 2.2} (b) holds. Let  $n$ be  large enough  so that $\mathscr{B}\subset \hat{\mathscr{D}}_n$. Applying Dynkin's formula [\cite{GH4}, Appendix C.3], for $ i\in \mathscr{B}^c\cap \hat{\mathscr{D}}_n$ we have
			\begin{align*}
			&E^{\pi^1,\pi^2}_i\biggl[e^{\int_{0}^{\tilde{\tau}(\mathscr{B})\wedge T\wedge \tau_n}\hat{\ell}(\xi_s)ds}V(\tilde{\tau}(\mathscr{B})\wedge T\wedge \tau_n)\biggr]-V(i)\nonumber\\
			&=E^{\pi^1,\pi^2}_i\biggl[\int_{0}^{\tilde{\tau}(\mathscr{B})\wedge T\wedge \tau_n}e^{\int_{0}^{t}\hat{\ell}(\xi_s)ds}[\hat{\ell}(\xi_t)V(\xi_t)+\sum_{j\in S}q(j|\xi_t,\pi^1(t),\pi^2(t))V(j)]dt\biggr]\nonumber\\
			&\leq E^{\pi^1,\pi^2}_i \biggl[\int_{0}^{\tilde{\tau}(\mathscr{B})\wedge T\wedge \tau_n}e^{\int_{0}^{t}\hat{\ell}(\xi_s)ds}C I_{\hat{\mathscr{K}}}(\xi_t)dt\biggr]=0,
			\end{align*} where $\tau_n=\inf\{t\geq 0:\xi_t\notin \hat{\mathscr{D}}_n\}$\,.
			Now by Fatou's lemma, taking first $n\rightarrow\infty$ and then $T\rightarrow\infty$, we get the required result. 			
			\end{proof}
	\end{lemma}
\begin{lemma}\label{lemm 2.3}
	Suppose Assumptions \ref{assm 2.1}, \ref{assm 2.2}, and \ref{assm 2.3} hold. Then for $n\in \mathbb{N}$, there exists a pair $(\rho_{n},\psi_{n})\in \mathbb{R}\times \mathcal{B}_{\hat{\mathscr{D}}_n}^{+}$, $\psi_{n}\gneq 0$ for the following Dirichlet nonlinear eigenequation
	\begin{align}
\rho_{n}\psi_{n}(i)&=\sup_{\mu\in P(A(i))}\inf_{\nu\in P(B(i))}\bigg[\sum_{j\in S}\psi_{n}(j)q(j|i,\mu,\nu)+c(i,\mu,\nu)\psi_{n}(i)\bigg]\nonumber\\
&=\inf_{\nu\in P(B(i))}\sup_{\mu\in P(A(i))}\bigg[\sum_{j\in S}\psi_{n}(j)q(j|i,\mu,\nu)+c(i,\mu,\nu)\psi_{n}(i)\bigg].\label{eq 2.24}
	\end{align}

	Also, for each $i\in S$ such that $\psi_{n}(i)>0$, we have
	\begin{align}
	\rho_{n}\leq \sup_{\pi^1\in \Pi^1}\inf_{\pi^2\in \Pi^2}\limsup_{T\rightarrow \infty}\frac{1}{T}\log E^{\pi^1,\pi^2}_i\bigg[e^{\int_{0}^{T}c(\xi_t,\pi^1(t),\pi^2(t))dt}\bigg].\label{eq 2.16}
	\end{align}
Additionally the sequence $\{\rho_{n}\}$ is bounded satisfying $\liminf_{n\rightarrow \infty}\rho_n\geq 0$.
	
\end{lemma}
\begin{proof}
	Let $\delta > 0$. Set $ \tilde{c} = \displaystyle{ c - \sup_{\hat{\mathscr{D}}_n}c - \delta}$\,. Let $\tilde{T}:\mathcal{B}_{\hat{\mathscr{D}}_n}\rightarrow\mathcal{B}_{\hat{\mathscr{D}}_n}$ be an operator defined as
	\begin{align}
	\tilde{T}(g)(i):=\sup_{\pi^1\in \Pi^1}\inf_{\pi^2\in \Pi^2} E^{\pi^1,\pi^2}_i\biggl[\int_{0}^{\tau_n}e^{\int_{0}^{t}\tilde{c}(\xi_s,\pi^1(s),\pi^2(s))ds}g(\xi_t)dt\biggr],~i\in \hat{\mathscr{D}}_n,\label{eq 2.17}
	\end{align} with $\tilde{T}(g)(i)=0~\text{ for }~i\in \hat{\mathscr{D}}_n^c$\,.
Let $g_1,g_2\in \mathcal{B}_{\hat{\mathscr{D}}_n}$ such that $g_1\succeq g_2$, i.e., $g_1(i) \geq g_2(i)$ for each $i$. Also, let $\tilde{T}(g_1) = \hat{\varphi}_1$ and $\tilde{T}(g_2) = \hat{\varphi}_2$\,. Then there exists $\hat{\pi}^{*1}\in \Pi^s_1$ such that 
\begin{align*}
-g_2(i)&=\sup_{\mu\in P(A(i))}\inf_{\nu\in P(B(i))}\biggl[\sum_{j\in S}\hat{\varphi}_2(j)q(j|i,\mu,\nu)+ \tilde{c}(i,\mu,\nu)\hat{\varphi}_2(i)\biggr]\nonumber\\
&=\inf_{\nu\in P(B(i))}\biggl[\sum_{j\in S}\hat{\varphi}_2(j)q(j|i,\hat{\pi}^{*1}(i),\nu)+ \tilde{c}(i,\hat{\pi}^{*1}(i),\nu)\hat{\varphi}_2(i)\biggr]~\forall i\in \hat{\mathscr{D}}_n\,.
\end{align*} Also, from the proof of Proposition \ref{prop 2.1}, we have
\begin{equation*}
\hat{\varphi}_2(i) = \tilde{T}(g_2)(i) = \inf_{\pi^2\in \Pi^2}E^{\hat{\pi}^{*1},\pi^2}_i\biggl[\int_{0}^{\tau_n}e^{\int_{0}^{t} \tilde{c}(\xi_s,\hat{\pi}^{*1}(\xi_s),\pi^2(s))ds}g_2(\xi_t)dt\biggr]\,.
\end{equation*} Thus, we deduce that
	\begin{align*}
\tilde{T}(g_1)(i)-	\tilde{T}(g_2)(i) = & \sup_{\pi^1\in \Pi^1}\inf_{\pi^2\in \Pi^2}E^{\pi^1,\pi^2}_i\biggl[\int_{0}^{\tau_n}e^{\int_{0}^{t} \tilde{c}(\xi_s,\pi^1(s),\pi^2(s))ds}g_1(\xi_t)dt\biggr]\nonumber\\
	& -\sup_{\pi^1\in \Pi^1}\inf_{\pi^2\in \Pi^2}E^{\pi^1,\pi^2}_i\biggl[\int_{0}^{\tau_n}e^{\int_{0}^{t} \tilde{c}(\xi_s,\pi^1(s),\pi^2(s))ds}g_2(\xi_t)dt\biggr]\\
	\geq & \inf_{\pi^2\in \Pi^2}E^{\hat{\pi}^{*1},\pi^2}_i\biggl[\int_{0}^{\tau_n}e^{\int_{0}^{t} \tilde{c}(\xi_s,\hat{\pi}^{*1}(\xi_s),\pi^2(s))ds}g_1(\xi_t)dt\biggr]\\
	&-\inf_{\pi^2\in \Pi^2}E^{\hat{\pi}^{*1},\pi^2}_i\biggl[\int_{0}^{\tau_n}e^{\int_{0}^{t} \tilde{c}(\xi_s,\hat{\pi}^{*1}(\xi_s),\pi^2(s))ds}g_2(\xi_t)dt\biggr]\\
 \geq & \inf_{\pi^2\in \Pi^2}E^{\hat{\pi}^{*1},\pi^2}_i\biggl[\int_{0}^{\tau_n}e^{\int_{0}^{t} \tilde{c}(\xi_s,\hat{\pi}^{*1}(\xi_s),\pi^2(s))ds}(g_1(\xi_t)-g_2(\xi_t))dt\biggr].
	\end{align*}
This gives us $\tilde{T}(g_1)\succeq \tilde{T}(g_2)$. Clearly $\tilde{T}(\lambda g) = \lambda \tilde{T}(g)$ for all $\lambda\geq 0$. Since $ \tilde{c} < -\delta$,  there exists a constant $\alpha_2>0$ such that $$\|\tilde{T}(\hat{g}_1) - \tilde{T}(\hat{g}_2)\|_{\hat{\mathscr{D}}_n}\leq \alpha_2 \|\hat{g}_1-\hat{g}_2\|_{\hat{\mathscr{D}}_n},~\text{ for any } \hat{g}_1, \hat{g}_2 \in \mathcal{B}_{\hat{\mathscr{D}}_n}.$$  Thus $\tilde{T}$ is continuous. Let $\{g_m\}$ be a bounded sequence in $\mathcal{B}_{\hat{\mathscr{D}}_n}$. Then from (\ref{eq 2.17}), for some constant $\alpha_3>0$ such that $\|\tilde{T}g_m\|_{\hat{\mathscr{D}}_n}\leq \alpha_3$. Now applying diagonalization arguments,  there exist a subsequence of $\{\tilde{T}g_m\}$, ( denoting by the same sequence without loss of generality) and a function $\phi\in \mathcal{B}_{\hat{\mathscr{D}}_n}$ such that $\|\tilde{T}g_m -\phi\|_{\hat{\mathscr{D}}_n}\rightarrow 0$ as $m\rightarrow\infty$. Hence the map $	\tilde{T}:\mathcal{B}_{\hat{\mathscr{D}}_n}\rightarrow\mathcal{B}_{\hat{\mathscr{D}}_n}$ is compact. Therefore it is completely continuous. Let $g\in \mathcal{B}_{\hat{\mathscr{D}}_n}$ such that $g(i_0)=1$ and $g(j)=0$ for all $j\neq i_0$. Then by (\ref{eq 2.17}), we have
	\begin{align*}
		\tilde{T}(g)(i_0)&\geq\sup_{\pi^1\in \Pi^1}\inf_{\pi^2\in \Pi^2}E^{\pi^1,\pi^2}_{i_0}\biggl[\int_{0}^{T_1}e^{\int_{0}^{t} \tilde{c}(\xi_s,\pi^1(s),\pi^2(s))ds}g(\xi_t)dt\biggr]\\
	&\geq \frac{g(i_0)}{\|\tilde{c}\|_{\hat{\mathscr{D}}_n}}\sup_{\pi^1\in \Pi^1}\inf_{\pi^2\in \Pi^2}E^{\pi^1,\pi^2}_{i_0}\bigg[1-e^{-\| \tilde{c}\|_{\hat{\mathscr{D}}_n}T_1}\bigg]\\
	&=g(i_0)\frac{1}{\| \tilde{c}\|_{\hat{\mathscr{D}}_n}+q^{*}(i_0)},
	\end{align*}
	where $T_1$ is the first jump time (clearly, $T_1\leq \tau_n$).
	Thus  $N\tilde{T}(g)\succeq g$ where $N={\| \tilde{c}\|_{\hat{\mathscr{D}}_n}+q^{*}(i_0)}>0$.
	Therefore by Theorem \ref{theo 2.1}, there exists a nontrivial $\psi_{n}\in \mathcal{B}_{\hat{\mathscr{D}}_n}^{+}$ where $\psi_{n}\neq 0$ and a constant $\lambda_{\hat{\mathscr{D}}_n}>0$ such that $	\tilde{T}(\psi_{n})=\lambda_{\hat{\mathscr{D}}_n}\psi_{n}$, i.e.,
	\begin{align*}
	\tilde{\rho}_{n}\psi_{n}(i)=\sup_{\mu\in P(A(i))}\inf_{\nu\in P(B(i))}\biggl[\psi_{n}(j)q(j|i,\mu,\nu)+ \tilde{c}(i,\mu,\nu)\psi_{n}(i)\biggr]~\forall i\in \hat{\mathscr{D}}_n,
	\end{align*}
	where $ \tilde{\rho}_{n}=-[\lambda_{\hat{\mathscr{D}}_n}]^{-1}$.
Therefore in terms of $c$, we have
\begin{align*}
	\rho_{n}\psi_{n}(i)=\sup_{\mu\in P(A(i))}\inf_{\nu\in P(B(i))}\biggl[\psi_{n}(j)q(j|i,\mu,\nu)+ c(i,\mu,\nu)\psi_{n}(i)\biggr]~\forall i\in \hat{\mathscr{D}}_n,
	\end{align*}
where $\rho_n = \tilde{\rho}_n + \displaystyle{ \sup_{\hat{\mathscr{D}}_n} c} + \delta$.
Now by Fan's minimax theorem, see [\cite{Fan}, Theorem 3], we have
	\begin{align*}
\rho_{n}\psi_{n}(i)&=\sup_{\mu\in P(A(i))}\inf_{\nu\in P(B(i))}\biggl[\psi_{n}(j)q(j|i,\mu,\nu)+c(i,\mu,\nu)\psi_{n}(i)\biggr]\\
&=\inf_{\nu\in P(B(i))}\sup_{\mu\in P(A(i))}\biggl[\psi_{n}(j)q(j|i,\mu,\nu)+c(i,\mu,\nu)\psi_{n}(i)\biggr]~\forall i\in \hat{\mathscr{D}}_n.
\end{align*}
This proves that (\ref{eq 2.24}) admits a unique solution. As before by the continuity of $c, q$ and the compactness of $A(i)$, there exists $\pi^{*1}_n \in \Pi^s_1$ such that (\ref{eq 2.24}), can be written as
	 \begin{align}
	 \rho_{n}\psi_{n}(i)=\inf_{\nu\in P(B(i))}\biggl[\psi_{n}(j)q(j|i,\pi^{*1}_n(i),\nu)+c(i,\pi^{*1}_n(i),\nu)\psi_{n}(i)\biggr]~\forall i\in \hat{\mathscr{D}}_n.\label{eq 2.19}
	 \end{align}
	Now applying Dynkin's formula (see [\cite{BP}, Lemma 3.1]) and using (\ref{eq 2.19}), we get
	\begin{align}
\psi_{n}(i)&\leq E^{\pi^{*1}_n,\pi^2}_i\biggl[e^{\int_{0}^{T}({c}(\xi_s,\pi^{*1}_n(\xi_s),\pi^2(s))-\rho_{n})ds}\psi_{n}(\xi_{T})I_{\{T<\tau_n\}}\biggr]\nonumber\\
&\leq (\sup_{\hat{\mathscr{D}}_n}\psi_{n})E^{\pi^{*1}_n,\pi^2}_i\biggl[e^{\int_{0}^{T}({c}(\xi_s,\pi^{*1}_n(\xi_s),\pi^2(s))-\rho_{n})ds}\biggr].\label{eq 2.20}
	\end{align}
		If $\psi_n(i)>0$ then by taking logarithm on the both sides in (\ref{eq 2.20}), dividing by $T$ and letting $T\rightarrow\infty$, we get
	\begin{align*}
\rho_{n}&\leq \limsup_{T\rightarrow \infty}\frac{1}{T}\log E^{\pi^{*1}_n,\pi^2}_i\biggl[e^{\int_{0}^{T} {c}(\xi_s,\pi^{*1}_n(\xi_s),\pi^2(s))ds}\biggr].
	\end{align*}
	Since $\pi^2\in \Pi^2$ is arbitrary, we obtain
		\begin{align*}
	\rho_{n}&\leq \inf_{\pi^2\in \Pi^2}\limsup_{T\rightarrow \infty}\frac{1}{T}\log E^{\pi^{*1}_n,\pi^2}_i\biggl[e^{\int_{0}^{T} {c}(\xi_s,\pi^{*1}_n(\xi_s),\pi^2(s))ds}\biggr]\\
	&\leq \sup_{\pi^1\in \Pi^1}\inf_{\pi^2\in \Pi^2}\limsup_{T\rightarrow \infty}\frac{1}{T}\log E^{\pi^{1},\pi^2}_i\biggl[e^{\int_{0}^{T}c(\xi_s,\pi^{1}(s),\pi^2(s))ds}\biggr].
	\end{align*}

	We now show that $J(i,c,\pi^1,\pi^2)$ is finite for every $(\pi^1,\pi^2)\in\Pi^1\times\Pi^2$ and $i\in S$.
	We only  provide a proof under Assumption 2.2 (b) and the proof under Assumption 2.2 (a) would be analogous. Now from (\ref{eq 2.6}) we get
	\begin{align}
	\sup_{(a,b)\in A(i)\times B(i)}\sum_{j\in S}V(j)q(j|i,a,b)\leq (C-\hat{\ell}(i))V(i) ~\forall i\in S.\label{eq 2.25}
	\end{align}
	Then by Dynkin formula, we get
	\begin{align}
	E^{\pi^1,\pi^2}_i\bigg[e^{\int_{0}^{T\wedge\tau_n}(\hat{\ell}(\xi_t)-C)dt}V(\xi_{T\wedge\tau_n})\bigg]\leq V(i)~\forall i\in S.\label{eq 2.26}
	\end{align}
By Fatou's lemma, taking $n\rightarrow\infty$ in (\ref{eq 2.26}), we get
	\begin{align*}
	E^{\pi^1,\pi^2}_i\bigg[e^{\int_{0}^{T}(\hat{\ell}(\xi_t)-C)dt}V(\xi_{T})\bigg]\leq V(i)~\forall i\in S.
	\end{align*}

	 Now, since $V\geq 1$, taking logarithm on both sides in the above equation, dividing both sides by $T$ and letting $T\rightarrow\infty $, we obtain
	$$J(i,\hat{\ell},\pi^1,\pi^2)\leq C~\text{ for all}~i\in S.$$
Since, $\hat{\ell}-\displaystyle \sup_{(a,b)\in A(i)\times B(i)}c(\cdot,a,b)$ is norm-like, we have 	$\displaystyle  \sup_{(a,b)\in A(i)\times B(i)}c(i,a,b)\leq \hat{\ell}(i)+k_1$ \, $\forall$ \,  $i\in S$ for some constant $k_1$. Hence we get
	\begin{align}
	J(i,c,\pi^1,\pi^2)\leq C+k_1~~\forall (\pi^1,\pi^2)\in \Pi^1\times \Pi^2, \forall i\in S.\label{eq 2.27}
	\end{align}
It is clear from (\ref{eq 2.16}) and (\ref{eq 2.27}) that $\rho_n$ has an upper bound. Next we prove that $\rho_{n}$ is bounded below. By using assumption \ref{assm 2.3} (iii) and (\ref{eq 2.24}), we have $\psi_n(i_0)>0$. Thus normalizing $\psi_{n}$, we have $\psi_{n}(i_0)=1$. Also, since $c\geq 0$, by (\ref{eq 2.24}) we get
	\begin{align*}
	\rho_{n}&\geq \sup_{\mu\in P(A(i_0))}\inf_{\nu\in P(B(i_0))}\bigg[\sum_{j\in S}\psi_{n}(j)q(j|i_0,\mu,\nu)\bigg]\\
	&\geq \sup_{\mu\in P(A(i_0))}\inf_{\nu\in P(B(i_0))}q(i_0|i_0,\mu,\nu).
	\end{align*}
	So, $\{\rho_{n}\}$ is bounded below.
	Now we claim that $\hat{\rho} := \displaystyle \liminf_{n\rightarrow \infty}\rho_{n}\geq 0$. If not, then on contrarary, $\hat{\rho}<0$. So, along some subsequence, we have (with an abuse of notation, we use the same sequence) $\rho_{n}\rightarrow\hat{\rho}$, as $n\rightarrow\infty$ and for large $n$, $\rho_n<0$. Let $\pi_n^{*2}$ be outer minimizing selector of (\ref{eq 2.24}). Thus, using (\ref{eq 2.24}), for large enough $n$, we have
	\begin{align*}
	0>\rho_{n}\psi_{n}(i_0)&=\sup_{\mu\in P(A(i_0))}\bigg[\sum_{j\in S}\psi_{n}(j)q(j|i_0,\mu,\pi^{*2}_n(i_0))+c(i_0,\mu,\pi_n^{*2}(i_0))\psi_{n}(i_0)\bigg]\\&\geq \sup_{\mu\in P(A(i_0))} \bigg[\sum_{j\in S}\psi_{n}(j)q(j|i_0,\mu,\pi^{*2}_n(i_0))\bigg]\\
	&\geq \sum_{j\in S}\psi_{n}(j)q(j|i_0,\mu,\pi_n^{*2}(i_0)).
	\end{align*}
	Now by Assumption \ref{assm 2.3} (iii), from the above equation, we get
	 \begin{align*}
	\psi_{n}(j)\leq \frac{-q(i_0|i_0,\mu,\pi_n^{*2}(i_0))}{q(j|i_0,\mu,\pi_n^{*2}(i_0))}\leq \sup_{\mu\in P(A(i_0))}\sup_{\nu\in P(B(i_0))}\frac{-q(i_0|i_0,\mu,\nu)}{q(j|i_0,\mu,\nu)}~\text{ for }~j\neq i_0.
	\end{align*}
	So, by diagonalization argument we say, there exist a subsequence (denoting by the same sequence with an abuse of notation) and a function $\psi$ with $\psi(i_0)=1$ such that $\psi_{n}(i)\rightarrow \psi(i)$, as $n\rightarrow \infty$ for all $i\in S$. By our assumption $A(i)$ is compact for each $i\in S$ and $\pi_n^{*2}$ is outer minimizing selector of (\ref{eq 2.24}). Hence we have  $\pi^{*2}_n(i)\rightarrow \pi^{*2}(i)$, for all $i\in S $, as $n \to \infty$. Therefore we have
	\begin{align}
	\rho_n\psi_{n}(i)
	&\geq \sum_{j\in S}\psi_{n}(j)q(j|i,\mu,\pi^{*2}_n(i))+c(i,\mu,\pi^{*2}_n(i))\psi_{n}(i).\label{eq 2.28}
	\end{align}
	So, taking $n\rightarrow \infty$ in the above equation, we obtain
	\begin{align}
	0>\hat{\rho}\psi(i)&\geq \sum_{j\in S}\psi(j)q(j|i,\mu,\pi^{*2}(i))+c(i,\mu,\pi^{*2}(i))\psi(i)\nonumber\\
	&\geq \sum_{j\in S}\psi(j)q(j|i,\mu,\pi^{*2}(i)).\label{eq 2.29}
	\end{align}
	Let $\pi^1\in \Pi^s_1$. Applying Dynkin formula and using (\ref{eq 2.29}), we obtain
	\begin{align*}
	&	E^{\pi^1,\pi^{*2}}_i[\psi(\xi_{t\wedge\tau_n})]-\psi(i)\\
	&=E^{\pi^1,\pi^{*2}}_i\biggl[\int_{0}^{t\wedge\tau_n}\sum_{j\in S}\psi(j)q(j|\xi_s,\pi^1(\xi_s),\pi^{*2}(\xi_s))ds\biggr]\leq 0.
	\end{align*}
	Now, using dominated convergence theorem, taking $n\rightarrow\infty$, we get $E^{\pi^1,\pi^{*2}}_i[\psi(\xi_{t})]\leq \psi(i)$.
	So, with respect to the canonical filtration of $\xi$, $\{\psi(\xi_t)\}$ is supermartingale. So, by Doob's martingale convergence theorem as $t\rightarrow\infty$, $\psi(\xi_t)$ converges. Now by Assumption \ref{assm 2.2}, $\xi$ is recurrent. Thus the skeleton process $\{\xi_n:n\in \mathbb{N}\}$ is also recurrent (see for details [\cite{A}, Proposition 5.1.1]). This implies, that the process $\{\xi_n:n\in \mathbb{N}\}$ visits every state of $S$ infinitely often. But this is possible only if $\psi\equiv 1$. Since $c\geq 0$, this contradicts (\ref{eq 2.29}). Thus, $\displaystyle \liminf_{n\rightarrow \infty}\rho_n\geq 0$.
	\end{proof}
\begin{lemma}\label{lemm 2.6}
	Suppose Assumptions \ref{assm 2.1}, \ref{assm 2.2}, and \ref{assm 2.3} hold. Then there exists $(\rho,\psi^{*})\in \mathbb{R}_+\times L^\infty_V$ with $\psi^{*}>0$, such that
	\begin{align}
	\rho\psi^{*}(i)&=\sup_{\mu\in P(A(i))}\inf_{\nu\in P(B(i))}\bigg[\sum_{j\in S}\psi^{*}(j)q(j|i,\mu,\nu)+c(i,\mu,\nu)\psi^{*}(i)\bigg]\nonumber\\
	&=\inf_{\nu\in P(B(i))}\sup_{\mu\in P(A(i))}\bigg[\sum_{j\in S}\psi^{*}(j)q(j|i,\mu,\nu)+c(i,\mu,\nu)\psi^{*}(i)\bigg], ~i\in S.\label{eq 2.30}
	\end{align}
	Also, the solution $(\rho,\psi^{*})$ has the following characteristic.
	\begin{enumerate}
		\item [(i)] $\rho\leq \displaystyle  \inf_{i\in S}\displaystyle  \displaystyle  \sup_{\pi^1\in \Pi^1}\inf_{\pi^2\in \Pi^2} \limsup_{T\rightarrow \infty}\frac{1}{ T}\ln E^{\pi^1,\pi^2}_i\biggl[e^{\int_{0}^{T} c(\xi_t,\pi^1(t),\pi^2(t))dt}\biggr]$.
		\item [(ii)] For any mini-max selector $(\pi^{*1},\pi^{*2})\in \Pi_1^s \times \Pi_2^s $ of (\ref{eq 2.30}), we have
	   \begin{align}
		\psi^{*}(i)&=\sup_{\pi^1\in \Pi^1}E^{\pi^1,\pi^{*2}}_i\bigg[e^{\int_{0}^{\tilde{\tau}(\mathscr{B})}(c(\xi_t,\pi^1(t),\pi^{*2}(\xi_t))-\rho)dt}\psi^{*}(\xi_{\tilde{\tau}(\mathscr{B})})\bigg]\nonumber\\
		&=\inf_{\pi^2\in\Pi^2}E^{\pi^{*1},\pi^2}_i\bigg[e^{\int_{0}^{\tilde{\tau}(\mathscr{B})}(c(\xi_t,\pi^{*1}(\xi_t),\pi^2(t))-\rho)dt}\psi^{*}(\xi_{\tilde{\tau}(\mathscr{B})})\bigg]~\forall i\in \mathscr{B}^c\,,\label{eq 2.31}
		\end{align}
	for some finite set $\mathscr{B}\supset \hat{\mathscr{K}}$.
	\end{enumerate}
\end{lemma}
\begin{proof}
	 Using Assumption \ref{assm 2.2} and the fact $c\geq 0$, there exists a finite set $\mathscr{B}$ containing $\hat{\mathscr{K}}$ such that the following hold.
	\begin{itemize}
		\item  When Assumption  \ref{assm 2.2} (a) holds:
		\begin{align}
		(\sup_{(a,b)\in A(i)\times B(i)}c(i,a,b)-\rho_{n})<\hat{\gamma}~\forall i\in \mathscr{B}^c,~\text{for all}~n~\text{large}.\label{eq 2.32}
		\end{align}
		\item  When Assumption  \ref{assm 2.2} (b) holds:
		\begin{align}
		(\sup_{(a,b)\in A(i)\times B(i)}c(i,a,b)-\rho_n)<\hat{\ell} (i)~\forall i\in \mathscr{B}^c,~\text{for all}~n~\text{large}.\label{eq 2.33}
		\end{align}
	\end{itemize}
Now we scale $\psi_n$ in such a way that it touches $V$ from below. Define
$$\hat{\theta}_n=\sup\{k>0:(V-k\psi_n)>0~\text{in}~S\}.$$
Then we see that $\hat{\theta}_n$ is finite as $\psi_{n}$ vanishes in $\hat{\mathscr{D}}_n^c$ and $\psi_{n}\gneq 0$. We claim that if we replace $\psi_{n}$ by $\hat{\theta}_n\psi_{n}$, then $\psi_{n}$ touches $V$ inside $\mathscr{B}$. If not, then  for some state $\hat{i}\in \mathscr{B}^c$, $(V-\psi_{n})(\hat{i})=0$ and $V-\psi_{n}>0$ in $\mathscr{B}\cup \hat{\mathscr{D}}_n^c$. Let $\pi^{*1}_n \in \Pi_{1}^s$ be an outer maximizing selector of (\ref{eq 2.24}). Then by Dynkin formula, we get (under Assumption \ref{assm 2.2} (b))
	\begin{align*}
\psi_n(\hat{i})&\leq E^{\pi^{*1}_n,\pi^2}_{\hat{i}}\biggl[e^{\int_{0}^{T\wedge\tilde{\tau}(\mathscr{B})}(c(\xi_s,\pi^{*1}_n(\xi_s),\pi^2(s))-\rho_n)ds}\psi_n(\xi_{T\wedge \tilde{\tau}(\mathscr{B})})I_{\{T\wedge\tilde{\tau}(\mathscr{B})<\tau_n\}}\biggr]\\
&\leq E^{\pi^{*1}_n,\pi^2}_{\hat{i}}\biggl[e^{\int_{0}^{T\wedge\tilde{\tau}(\mathscr{B})}\hat{\ell}(\xi_s)ds}\psi_n(\xi_{T\wedge \tilde{\tau}(\mathscr{B})})I_{\{T\wedge\tilde{\tau}(\mathscr{B})<\tau_n\}}\biggr].
\end{align*}
 Since $\psi_n\leq V$, in view of Lemma \ref{lemm 2.4}, by the dominated convergence theorem, taking $T\rightarrow\infty$, we get
	\begin{align*}
\psi_n(\hat{i})
\leq E^{\pi^{*1}_n,\pi^2}_{\hat{i}}\biggl[e^{\int_{0}^{\tilde{\tau}(\mathscr{B})}\hat{\ell}(\xi_s)ds}\psi_n(\xi_{ \tilde{\tau}(\mathscr{B})})\biggr].
\end{align*}
 Using this and (\ref{eq 2.22}), we have
	\begin{align*}
0=(V-\psi_n)(\hat{i})\geq E^{\pi^{*1}_n,\pi^2}_{\hat{i}}\bigg[e^{\int_{0}^{\tilde{\tau}(\mathscr{B})}\hat{\ell}(\xi_s) ds}(V-\psi_n)(\xi_{\tilde{\tau}(\mathscr{B})})\bigg]>0.
\end{align*}
Hence we arrive at a contradiction. Thus $\psi_n$ touches $V$ inside $\mathscr{B}$. Similar conclusion holds under Assumption \ref{assm 2.2} (a). Now, since $\psi_{n}\leq V$ for all large n, by diagonalization argument, there exists a subsequence (by an abuse of notation, we use the same sequence) such that, $\psi_{n}\rightarrow\psi^{*}$ for all $i\in S$, as $n\rightarrow \infty$, and $\psi^{*}\leq V$. Also, since by Lemma \ref{lemm 2.3}, the sequence $\{\rho_{n}\}$ is bounded and $\displaystyle \liminf_{n\rightarrow \infty}\rho_{n}\geq 0$, we can find a subsequence (by an abuse of notation we use the same sequence) and some $\rho\geq 0$ such that $\rho_{n}\rightarrow \rho$ as $n\rightarrow \infty$.
Thus as before, there exists a mini-max selector $(\pi^{*1}_n, \pi^{*2}_n)\in \Pi_1^s\times \Pi_2^s$ of (\ref{eq 2.24}), i.e.,
\begin{align}
&\inf_{\nu\in P(B(i))}\bigg[\sum_{j\in S}\psi_{n}(j)q(j|i,\pi^{*1}_n(i),\nu)+c(i,\pi^{*1}_n(i),\nu)\psi_{n}(i)\bigg]\nonumber\\
& = \sup_{\mu\in P(A(i))}\inf_{\nu\in P(B(i))}\bigg[\sum_{j\in S}\psi_{n}(j)q(j|i,\mu,\nu)+c(i,\mu,\nu)\psi_{n}(i)\bigg]\nonumber\\
& = \inf_{\nu\in P(B(i))}\sup_{\mu\in P(A(i))}\bigg[\sum_{j\in S}\psi_{n}(j)q(j|i,\mu,\nu)+c(i,\mu,\nu)\psi_{n}(i)\bigg]\nonumber\\
& =\sup_{\mu\in P(A(i))}\bigg[\sum_{j\in S}\psi_{n}(j)q(j|i,\mu,\pi^{*2}_n(i))+c(i,\mu,\pi^{*2}_n(i))\psi_{n}(i)\bigg].\label{L2.3E2.31}
\end{align}
Hence,
\begin{align*}
\rho_{n}\psi_{n}(i)\leq \bigg[\sum_{j\in S}\psi_{n}(j)q(j|i,\pi^{*1}_n(i),\nu)+c(i,\pi^{*1}_n(i),\nu)\psi_{n}(i)\bigg].
\end{align*}
The above implies
\begin{align}
\rho_{n}\psi_{n}(i)-\psi_{n}(i)q(i|i,\pi^{*1}_n(i),\nu)\leq \bigg[\sum_{j\neq  i}\psi_{n}(j)q(j|i,\pi^{*1}_n(i),\nu)+c(i,\pi^{*1}_n(i),\nu)\psi_{n}(i)\bigg].\label{L2.6E2.34}
\end{align}
Now, since $\psi_{n}(i)\leq V(i)$ for all $i\in S$, we have
\begin{align}
\sum_{j\neq i}\psi(j)q(j|i,\pi^{*1}_n(i),\nu)\leq \sum_{j\neq i} V(j) q(j|i,\pi^{*1}_n(i),\nu).\label{L2.6E2.35}
\end{align}
Also, since $\Pi_1^s$ and $\Pi_2^{s}$ are compact there exist $\pi^{*1}\in \Pi_1^s$ and $\pi^{*2}\in \Pi_2^s$ such that $\pi^{*1}_n\rightarrow\pi^{*1}$ and $\pi^{*2}_n\rightarrow\pi^{*2}$ as $n\rightarrow\infty$. Under given assumptions, from [\cite{GH1}, Lemma 7.2] it is clear that the functions $c(i,\mu,\nu)$, and $\displaystyle \sum_{j\in S}q(j|i,\mu,\nu)u(j)$ are continuous at $(\mu,\nu)$ on $P(A(i))\times P(B(i))$ for each fixed $u \in L^\infty_V$, $i\in S$.
Therefore by the dominated convergence theorem, letting $n\rightarrow\infty$ in (\ref{L2.6E2.34}), we obtain
\begin{align*}
\rho\psi^{*}(i)\leq \sum_{j\in S}\psi^{*}(j)q(j|i,\pi^{*1}(i),\nu)+c(i,\pi^{*1}(i),\nu)\psi^{*}(i).
\end{align*}
Hence we have
\begin{align}
\rho\psi^{*}(i)&\leq \inf_{\nu\in P(B(i))} \bigg[\sum_{j\in S}\psi^{*}(j)q(j|i,\pi^{*1}(i),\nu)+c(i,\pi^{*1}(i),\nu)\psi^{*}(i)\biggr].\nonumber\\
&\leq \sup_{\mu\in P(A(i))}\inf_{\nu\in P(B(i))} \bigg[\sum_{j\in S}\psi^{*}(j)q(j|i,\mu,\nu)+c(i,\mu,\nu)\psi^{*}(i)\biggr].\label{eq 2.34}
\end{align}
By similar arguments using (\ref{L2.3E2.31}) and extended Fatou's lemma [\cite{HL}, Lemma 8.3.7], we get
\begin{align}
\rho\psi^{*}(i)&\geq  \sup_{\mu\in P(A(i))}\bigg[\sum_{j\in S}\psi^{*}(j)q(j|i,\mu,\pi^{*2}(i))+c(i,\mu,\pi^{*2}(i))\psi^{*}(i)\biggr]\nonumber\\
&\geq \inf_{\nu\in P(B(i))} \sup_{\mu\in P(A(i))} \bigg[\sum_{j\in S}\psi^{*}(j)q(j|i,\mu,\nu)+c(i,\mu,\nu)\psi^{*}(i)\biggr].\label{eq 2.35}
\end{align}
Hence by (\ref{eq 2.34}) and (\ref{eq 2.35}), we get (\ref{eq 2.30}).
Since at some point in $\mathscr{B}$ we have $(V-\psi_{n})=0$, for all large $n$. It follows that $(V-\psi^{*})(i^*)=0$ for some $i^*\in\mathscr{B}$. Since $V\geq 1$, it is clear that $\psi^{*}$ is nontrivial. Now we claim that $\psi^{*}>0$. If not, then we must have $\psi^{*}(\tilde{i})=0$ for some $\tilde{i}\in S$.
Again as before, there exits a pair of a mini-max selector $(\pi^{*1},\pi^{*2})\in \Pi_1^s\times \Pi_2^s$ such that form (\ref{eq 2.30}), we have
\begin{align}
\rho\psi^{*}(\tilde{i})=\bigg[\sum_{j\in S}\psi^{*}(j)q(j|\tilde{i},\pi^{*1}(\tilde{i}),\pi^{*2}(\tilde{i}))+c(\tilde{i},\pi^{*1}(\tilde{i}),\pi^{*2}(\tilde{i}))\psi^{*}(\tilde{i})\biggr].\label{eq 2.37}
\end{align}
This implies
$$\sum_{j\neq \tilde{i}}\psi^{*}(j)q(j|\tilde{i},\pi^{*1}(\tilde{i}),\pi^{*2}(\tilde{i}))=0.$$
Since the Markov chain $\xi$ is irreducible under $(\pi^{*1},\pi^{*2})\in \Pi^s_1\times\Pi^s_2$, from the above equation, it follows that $\psi^{*}\equiv 0$. So, we arrive at a contradiction. This proves the claim.
Now we prove (i) and (ii).

	(i)  Since $\psi^{*}>0$ and $\psi_n(i)\rightarrow \psi^{*}(i)$ as $n\rightarrow \infty$, we have $\psi_{n}>0$ for all large enough $n$. So, using (\ref{eq 2.16}), we have $\displaystyle    \lim_{n\rightarrow \infty}\rho_{n}\leq \sup_{\pi^1\in \Pi^1}\inf_{\pi^2\in \Pi^2}J(i,c,\pi^1,\pi^2)$ for all $i\in S$.
	
(ii) By measurable selection theorem in [\cite{N}, Theorem 2.2], there exists a pair of strategies (a mini-max selector) $(\pi^{*1},\pi^{*2})\in \Pi_1^s\times \Pi_2^s$ (as in (\ref{L2.3E2.31})) satisfying
\begin{align}
\rho\psi^{*}(i)&= \sup_{\mu\in P(A(i))}\bigg[\sum_{j\in S}\psi^{*}(j)q(j|i,\mu,\pi^{*2}(i))+c(i,\mu,\pi^{*2}(i))\psi^{*}(i)\biggr]\nonumber\\
&=\inf_{\nu\in P(B(i))} \bigg[\sum_{j\in S}\psi^{*}(j)q(j|i,\pi^{*1}(i),\nu)+c(i,\pi^{*1}(i),\nu)\psi^{*}(i)\biggr].\label{eq 2.38}
\end{align}	
 Using (\ref{eq 2.38}), Lemma \ref{lemm 2.4}, and Dynkin's formula, we have
\begin{align*}
\psi^{*}(i)\geq E^{\pi^1,\pi^{*2}}_i\bigg[e^{\int_{0}^{\tilde{\tau}(\mathscr{B})\wedge T}(c(\xi_t,\pi^1(t),\pi^{*2}(\xi_t))-\rho)dt}\psi^{*}(\xi_{\tilde{\tau}(\mathscr{B})\wedge T})\bigg]~\forall i\in \mathscr{B}^c.
\end{align*}
By Fatou's lemma taking $T\rightarrow\infty$, we get
\begin{align}
\psi^{*}(i)\geq E^{\pi^1,\pi^{*2}}_i\bigg[e^{\int_{0}^{\tilde{\tau}(\mathscr{B})}(c(\xi_t,\pi^1(t),\pi^{*2}(\xi_t))-\rho)dt}\psi^{*}(\xi_{\tilde{\tau}(\mathscr{B})})\bigg],~\forall i\in \mathscr{B}^c.\label{eq 2.39}
\end{align}
Hence,
\begin{align}
\psi^{*}(i)&\geq\sup_{\pi^1\in \Pi^1} E^{\pi^1,\pi^{*2}}_i\bigg[e^{\int_{0}^{\tilde{\tau}(\mathscr{B})}(c(\xi_t,\pi^1(t),\pi^{*2}(\xi_t))-\rho)dt}\psi^{*}(\xi_{\tilde{\tau}(\mathscr{B})})\bigg]\nonumber\\
&\geq \inf_{\pi^2\in\Pi^2}\sup_{\pi^1\in \Pi^1}E^{\pi^1,\pi^{2}}_i\bigg[e^{\int_{0}^{\tilde{\tau}(\mathscr{B})}(c(\xi_t,\pi^1(t),\pi^{2}(t))-\rho)dt}\psi^{*}(\xi_{\tilde{\tau}(\mathscr{B})})\bigg],~\forall i\in \mathscr{B}^c.\label{eq 2.40}
\end{align}
Also, using (\ref{eq 2.38}), Lemma \ref{lemm 2.4}, and Dynkin's formula, we obtain
\begin{align*}
\psi^{*}(i)\leq E^{\pi^{*1},\pi^{2}}_i\bigg[e^{\int_{0}^{\tilde{\tau}(\mathscr{B})\wedge T}(c(\xi_t,\pi^{*1}(\xi_t),\pi^{2}(t))-\rho)dt}\psi^{*}(\xi_{\tilde{\tau}(\mathscr{B})\wedge T})\bigg]~\forall i\in \mathscr{B}^c.
\end{align*}
Since $\psi^{*}\leq V$, using the estimates as in Lemma \ref{lemm 2.4}, taking $T\rightarrow\infty$, by dominated convergent theorem it follows that
\begin{align}
\psi^{*}(i)&\leq  E^{\pi^{*1},\pi^{2}}_i\bigg[e^{\int_{0}^{\tilde{\tau}(\mathscr{B})}(c(\xi_t,\pi^{*1}(\xi_t),\pi^{2}(t))-\rho)dt}\psi^{*}(\xi_{\tilde{\tau}(\mathscr{B})})\bigg]~\forall i\in \mathscr{B}^c.\label{eq 2.41}
\end{align}
Hence
\begin{align}
\psi^{*}(i)&\leq \inf_{\pi^2\in \Pi^2} E^{\pi^{*1},\pi^{2}}_i\bigg[e^{\int_{0}^{\tilde{\tau}(\mathscr{B})}(c(\xi_t,\pi^{*1}(\xi_t),\pi^{2}(t))-\rho)dt}\psi^{*}(\xi_{\tilde{\tau}(\mathscr{B})})\bigg]\nonumber\\
&\leq  \sup_{\pi^1\in \Pi^1}\inf_{\pi^2\in\Pi^2}E^{\pi^1,\pi^{2}}_i\bigg[e^{\int_{0}^{\tilde{\tau}(\mathscr{B})}(c(\xi_t,\pi^1(t),\pi^{2}(t))-\rho)dt}\psi^{*}(\xi_{\tilde{\tau}(\mathscr{B})})\bigg],~\forall i\in \mathscr{B}^c.\label{eq 2.42}
\end{align}
From (\ref{eq 2.40}) and (\ref{eq 2.42}), we get (\ref{eq 2.31}).
	\end{proof}
\section{Existence of risk-sensitive average optimal strategies}
In this section we prove that any mini-max selector of the associated HJI equation is a saddle point equilibrium. Also, exploiting the stochastic representation (\ref{eq 2.31}) we completely characterize all possible saddle point equilibrium in the space of stationary Markov strategies.
\begin{thm}\label{theo 3.1}
Suppose Assumptions \ref{assm 2.1}, \ref{assm 2.2}, and \ref{assm 2.3} hold. Then for any mini-max selector $(\pi^{*1},\pi^{*2})\in \Pi^s_1\times \Pi^s_2$ of (\ref{eq 2.30}), i.e., for any pair $(\pi^{*1},\pi^{*2})\in \Pi^s_1\times \Pi^s_2$ satisfying
\begin{align}
&\inf_{\nu\in P(B(i))}\bigg[\sum_{j\in S}\psi^{*}(j)q(j|i,\pi^{*1}(i),\nu)+c(i,\pi^{*1}(i),\nu)\psi^{*}(i)\bigg]\nonumber\\
&= \sup_{\mu\in P(A(i))}\inf_{\nu\in P(B(i))}\bigg[\sum_{j\in S}\psi^{*}(j)q(j|i,\mu,\nu)+c(i,\mu,\nu)\psi^{*}(i)\bigg]\nonumber\\
&=\inf_{\nu\in P(B(i))}\sup_{\mu\in P(A(i))}\bigg[\sum_{j\in S}\psi^{*}(j)q(j|i,\mu,\nu)+c(i,\mu,\nu)\psi^{*}(i)\bigg]\nonumber\\
&=\sup_{\mu\in P(A(i))}\bigg[\sum_{j\in S}\psi^{*}(j)q(j|i,\mu,\pi^{*2}(i))+c(i,\mu,\pi^{*2}(i))\psi^{*}(i)\bigg], ~i\in S,\label{eq 2.46}
\end{align}
we have
\begin{align}
\rho&=\inf_{i\in S}\sup_{\pi^1\in \Pi^1}\limsup_{T\rightarrow \infty}\frac{1}{T}\log E^{\pi^1,\pi^{*2}}_i\bigg[e^{\int_{0}^{T}c(\xi_t,\pi^1(t),\pi^{*2}(\xi_t))dt}\bigg]\nonumber\\
&=\inf_{i\in S}\sup_{\pi^1\in \Pi^1}\inf_{\pi^2\in \Pi^2}\limsup_{T\rightarrow \infty}\frac{1}{T}\log E^{\pi^{1},\pi^{2}}_i\bigg[e^{\int_{0}^{T}c(\xi_t,\pi^{1}(t),\pi^{2}(t))dt}\bigg]\nonumber\\
&=\inf_{i\in S}\inf_{\pi^2\in \Pi^2}\sup_{\pi^1\in \Pi^1}\limsup_{T\rightarrow \infty}\frac{1}{T}\log E^{\pi^{1},\pi^{2}}_i\bigg[e^{\int_{0}^{T}c(\xi_t,\pi^{1}(t),\pi^{2}(t))dt}\bigg].
\label{eq 2.47}
\end{align}
\end{thm}
\begin{proof}
We perturb the cost function as follows.
	\begin{itemize}
		\item 	If Assumption \ref{assm 2.2} (a) holds: We define for $ (a,b)\in A(i)\times B(i)$, $i\in S$,
		$\hat{c}_n(i,a,b)=c(i,a,b)I_{\hat{\mathscr{D}}_n}(i)+(\|c\|_\infty+\alpha_3)I_{\hat{\mathscr{D}}_n^c}$. Here $\alpha_3>0$, is a small number satisfying $\|c\|_\infty+\alpha_3<\hat{\gamma}$.
		Note that $\|\hat{c}_n\|_\infty<\hat{\gamma}$.
		\item If Assumption \ref{assm 2.2} (b) holds: We define for $(a,b)\in A(i)\times B(i)$, $i\in S$,
		$\hat{c}_n(i,a,b)=c(i,a,b)+\frac{1}{n}[\hat{\ell}(i)-\displaystyle \sup_{(a,b)\in A(i)\times B(i)}c(i,a,b)]_+$.
		Note that the function $[\hat{\ell}(\cdot)-\displaystyle \sup_{(a,b)\in A(\cdot)\times B(\cdot)}c(\cdot,a,b)]_+$  is norm-like function. Also, it is easy to see that for large enough $n$, $\hat{\ell}(\cdot)-\displaystyle \sup_{(a,b)\in A(\cdot)\times B(\cdot)}\hat{c}_n(\cdot,a,b)$ is norm-like.
	\end{itemize}
In view of Lemma \ref{lemm 2.6}, it is clear that for $\pi^{*2}\in \Pi^s_2$, there exists $(\tilde{\psi}_{n},\tilde{\rho}_{n})\in  L^\infty_V \times\mathbb{R}_{+}$, $\tilde{\psi}_{n}>0$ satisfying
	\begin{align}
	{\tilde{\rho}_{n}}\tilde{\psi}_{n}(i)=\sup_{\mu\in P(A(i))}\bigg[\sum_{j\in S}\tilde{\psi}_{n}(j)q(j|i,\mu,\pi^{*2}(i))+{\hat{c}_n(i,\mu,\pi^{*2}(i))}\tilde{\psi}_{n}(i)\bigg]\label{eq 2.48}
	\end{align}
	such that
	\begin{align}
	0\leq\tilde{\rho}_{n}\leq \sup_{\pi^1\in\Pi^1}\limsup_{T\rightarrow \infty}\frac{1}{T}\log E^{\pi^{1},\pi^{*2}}_i\bigg[e^{\int_{0}^{T}\hat{c}_n(\xi_t,\pi^{1}(t),\pi^{*2}(\xi_t))dt}\bigg].\label{eq 2.49}
	\end{align}
	Also, for some finite set $\mathscr{B}_1 \supset \mathscr{B} \supset \mathscr{K}$\,, we have
	\begin{align}
\tilde{\psi}_{n}(i)=\sup_{\pi^1\in\Pi^1}E^{\pi^{1},\pi^{*2}}_i\bigg[e^{\int_{0}^{\tilde{\tau}({\mathscr{B}_1})}(\hat{c}_n(\xi_t,\pi^{1}(t),\pi^{*2}(\xi_t))-\tilde{\rho}_{n})dt}\tilde{\psi}_{n}(\xi_{\tilde{\tau}(\mathscr{B}_1)})\bigg],~i\in \mathscr{B}_1^c.\label{eq 2.45}
	\end{align}
Now from the proof of Lemma \ref{lemm 2.6}, we have a finite set $\tilde{\mathscr{B}}$, depending on $n$, containing $\hat{\mathscr{K}}$ such that the following cases happen:
\begin{itemize}
	\item Under Assumption \ref{assm 2.2} (a): From (\ref{eq 2.49}), we have $\tilde{\rho}_{n}\leq \|\hat{c}_n\|_\infty$. Thus, for $i\in {\hat{\mathscr{D}}_n}^c$, it follows that $\hat{c}_n(i,a,b)-\tilde{\rho}_{n}\geq 0$ for all $(a,b)\in A(i)\times B(i)$. Consequently, we may take $\tilde{\mathscr{B}}=\hat{\mathscr{D}}_n$ such that $\hat{c}_n(i,a,b)-\tilde{\rho}_{n}\geq 0$ in $\tilde{\mathscr{B}}^c$ for all $(a,b)\in A(i)\times B(i)$.
	\item  Under Assumption \ref{assm 2.2} (b): since $\hat{c}_n$ is norm-like function, we can choose suitable finite set $\tilde{\mathscr{B}}$ such that $(\hat{c}_n(i,a,b)-\tilde{\rho}_{n})\geq 0$ in $\tilde{\mathscr{B}}^c$ for all $(a,b)\in A(i)\times B(i)$.
\end{itemize}
For any $\pi^1\in\Pi^1$, applying Dynkin formula and using (\ref{eq 2.48}) and Lemma \ref{lemm 2.4}, we get
\begin{align*}
\tilde{\psi}_{n}(i)\geq E^{\pi^{1},\pi^{*2}}_i\bigg[e^{\int_{0}^{\tilde{\tau}(\tilde{\mathscr{B}})\wedge T}(\hat{c}_n(\xi_t,\pi^{1}(t),\pi^{*2}(\xi_t))-\tilde{\rho}_{n})dt}\tilde{\psi}_{n}(\xi_{\tilde{\tau}(\tilde{\mathscr{B}})\wedge T})\bigg].
\end{align*}
Since for $i\in \tilde{\mathscr{B}}^c$, $\hat{c}_n(i,a,b)-\tilde{\rho}_{n}\geq 0$, by Fatou's lemma taking $T\rightarrow \infty$, we get
		\begin{align*}
\tilde{\psi}_{n}(i)&\geq E^{\pi^{1},\pi^{*2}}_i\bigg[e^{\int_{0}^{\tilde{\tau}(\tilde{\mathscr{B}})}(\hat{c}_n(\xi_t,\pi^{1}(t),\pi^{*2}(\xi_t))-\tilde{\rho}_{n})dt}\tilde{\psi}_{n}
(\xi_{\tilde{\tau}(\tilde{\mathscr{B}})})\bigg]\\
	&\geq (\min_{\tilde{\mathscr{B}}}\tilde{\psi}_{n})~ \forall ~i\in \tilde{\mathscr{B}}^c.
	\end{align*}
This implies that, $\tilde{\psi}_{n}$ has a lower bound. Now, applying Dynkin formula, and using (\ref{eq 2.48}) and Lemma \ref{lemm 2.4}, we deduce that
		\begin{align*}
\tilde{\psi}_{n}(i)&\geq E^{\pi^{1},\pi^{*2}}_i\bigg[e^{\int_{0}^{T\wedge \tau_N}(\hat{c}_n(\xi_t,\pi^{1}(t),\pi^{*2}(\xi_t))-\tilde{\rho}_{n})dt}\tilde{\psi}_{n}(\xi_{T\wedge\tau_N})\bigg],
	\end{align*}
	for any $i\in S$, where $\tau_N:=\inf\{t\geq 0: \xi_t\notin \{1,2,\cdots, N\}\}$, $N\in \mathbb{N}$. By Fatou's lemma taking $N\rightarrow \infty$, we get
	\begin{align*}
\tilde{\psi}_{n}(i)&\geq E^{\pi^{1},\pi^{*2}}_i\bigg[e^{\int_{0}^{T}(\hat{c}_n(\xi_t,\pi^{1}(t),\pi^{*2}(\xi_t))-\tilde{\rho}_{n})dt}\tilde{\psi}_{n}(\xi_{T})\bigg]\\
&\geq \min_{\tilde{\mathscr{B}}}\tilde{\psi}_{n} E^{\pi^{1},\pi^{*2}}_i \bigg[e^{\int_{0}^{T}(\hat{c}_n(\xi_t,\pi^{1}(t),\pi^{*2}(\xi_t))-\tilde{\rho}_{n})dt}\bigg].
\end{align*}
Thus, taking logarithm both sides, dividing by $T$ and letting $T\to\infty$, we obtain
		\begin{align*}
\tilde{\rho}_{n} \geq \limsup_{T\rightarrow \infty}\frac{1}{T}\log E^{\pi^{1},\pi^{*2}}_i\bigg[e^{\int_{0}^{T}\hat{c}_n(\xi_t,\pi^{1}(t),\pi^{*2}(\xi_t))dt}\bigg].
	\end{align*} Since $\pi^1\in\Pi^1$ arbitrary, it follows that
	\begin{align*}
\tilde{\rho}_{n}&\geq \sup_{\pi^1\in\Pi^1}\limsup_{T\rightarrow \infty}\frac{1}{T}\log E^{\pi^{1},\pi^{*2}}_i\bigg[e^{\int_{0}^{T}\hat{c}_n(\xi_t,\pi^{1}(t),\pi^{*2}(\xi_t))dt}\bigg]\\
&\geq \sup_{\pi^1\in\Pi^1}\limsup_{T\rightarrow \infty}\frac{1}{T}\log E^{\pi^{1},\pi^{*2}}_i\bigg[e^{\int_{0}^{T}{c}(\xi_t,\pi^{1}(t),\pi^{*2}(\xi_t))dt}\bigg].
	\end{align*}
	Using this and (\ref{eq 2.49}), we get $\displaystyle \sup_{\pi^1\in\Pi^1}J(i,c,\pi^{1},\pi^{*2})\leq\sup_{\pi^1\in\Pi^1} J(i,\hat{c}_n,\pi^{1},\pi^{*2})=\tilde{\rho}_{n}$ for all $n$. From the definition of $\hat{c}_n$, it is easy to see that $\tilde{\rho}_{n}$ is a decreasing sequence which has a lower bound. Now by similar arguments as in Lemma \ref{lemm 2.6}, it follows that there exists a pair $(\tilde{\rho},\tilde{\psi})$ such that $\tilde{\rho}_{n}\rightarrow \tilde{\rho}$ and $\tilde{\psi}_{n}\rightarrow\tilde{\psi}$ as $n\rightarrow\infty$. As in Lemma \ref{lemm 2.6}, by taking $n\rightarrow\infty$ in (\ref{eq 2.48}), we get
	\begin{align}
{\tilde{\rho}}\tilde{\psi}(i)=\sup_{\mu\in P(A(i))}\bigg[\sum_{j\in S}\tilde{\psi}(j)q(j|i,\mu,\pi^{*2}(i))+{{c}(i,\mu,\pi^{*2}(i))}\tilde{\psi}(i)\bigg].\label{eq 2.50}
\end{align}
Also, we have $\displaystyle \tilde{\rho}\geq \sup_{\pi^1\in \Pi^1}J(i,c,\pi^{1},\pi^{*2})\geq \rho$. Now, we want to show that $\tilde{\rho}=\rho$. Let $\tilde{\pi}^{*1}$ be a selector in (\ref{eq 2.50}). Thus
	\begin{align}
{\tilde{\rho}}\tilde{\psi}(i)=\bigg[\sum_{j\in S}\tilde{\psi}(j)q(j|i,\tilde{\pi}^{*1}(i),\pi^{*2}(i))+{{c}(i,\tilde{\pi}^{*1}(i),\pi^{*2}(i))}\tilde{\psi}(i)\bigg].\label{T3.1E3.9}
\end{align}
In view of estimates  in Lemma \ref{lemm 2.4}, applying Dynkin's formula and the dominated convergence theorem, from (\ref{T3.1E3.9}) we deduce that there exists a finite set $\mathscr{B}_2 \supset \mathscr{B}_1$ such that
	\begin{align}
\tilde{\psi}(i) \leq E^{\tilde{\pi}^{*1},\pi^{*2}}_i\bigg[e^{\int_{0}^{\tilde{\tau}({\mathscr{B}_2})}({c}(\xi_t,\tilde{\pi}^{*1}(\xi_t),\pi^{*2}(\xi_t))-\tilde{\rho})dt}\tilde{\psi}(\xi_{\tilde{\tau}(\mathscr{B}_2)})\bigg],~\forall i\in \mathscr{B}_2^c.\label{eq 2.51}
\end{align}

Since $\tilde{\rho}\geq \rho$, arguing as in Lemma \ref{lemm 2.6} (see, (\ref{eq 2.31})) for $\mathscr{B}_2 \supset \mathscr{B}$ we have
\begin{align}
\psi^{*}(i)\geq E^{\tilde{\pi}^{*1},\pi^{*2}}_i\bigg[e^{\int_{0}^{\tilde{\tau}({\mathscr{B}_2})}({c}(\xi_t,\tilde{\pi}^{*1}(\xi_t),\pi^{*2}(\xi_t))-\tilde{\rho})dt}\psi^{*}(\xi_{\tilde{\tau}(\mathscr{B}_2)})\bigg]~\forall i\in \mathscr{B}_2^c.\label{eq 2.52}
\end{align}
Now we choose an appropriate constant $\kappa$ (e.g., $\displaystyle \kappa = \min_{\mathscr{B}_2}\frac{\psi^{*}}{\tilde{\psi}}$), so that $(\psi^{*}-\kappa\tilde{\psi}) \geq 0$ in $\mathscr{B}_2$ and for some $\hat{i}_0\in \mathscr{B}_2,$ \,$(\psi^{*}-\kappa \tilde{\psi})(\hat{i}_0) = 0$.
From  (\ref{eq 2.51}) and (\ref{eq 2.52}), we get
\begin{align}
\psi^{*}(i) - \kappa \tilde{\psi}(i)\geq E^{\tilde{\pi}^{*1},\pi^{*2}}_i\bigg[e^{\int_{0}^{\tilde{\tau}({\mathscr{B}_2})}({c}(\xi_t,\tilde{\pi}^{*1}(\xi_t),\pi^{*2}(\xi_t))-\tilde{\rho})dt}(\psi^{*} - \kappa\tilde{\psi})(\xi_{\tilde{\tau}(\mathscr{B}_2)})\bigg]~\forall i\in \mathscr{B}_2^c.\label{eq 2.53}
\end{align} From the above expression it is easy to see that $(\psi^{*}-\kappa\tilde{\psi}) \geq 0$ in $S$. Now using (\ref{eq 2.46}), (\ref{T3.1E3.9}) and the fact that $\tilde{\rho}\geq  \rho$, we get 
  \begin{align*}
\tilde{\rho}(\psi^{*} &- \kappa \tilde{\psi})(\hat{i}_0) \\
& \geq \biggl[\sum_{j\in S}(\psi^{*} - \kappa\tilde{\psi})(j)q(j|\hat{i}_0 ,\tilde{\pi}^{*1}(\hat{i}_0),\pi^{*2}(\hat{i}_0)) + c(\hat{i}_0,\tilde{\pi}^{*1}(\hat{i}_0),\pi^{*2}(\hat{i}_0))(\psi^{*} - \kappa\tilde{\psi})(\hat{i}_0)\biggr].	
\end{align*}
This implies that
\begin{equation}\label{eq 2.54}
\sum_{j\neq \hat{i}_0}(\psi^{*} - \kappa\tilde{\psi})(j)q(j|{\hat{i}_0},\tilde{\pi}^{*1}(\hat{i}_0),\pi^{*2}(\hat{i}_0)) = 0\,.
\end{equation}
Since the Markov chain $\xi$ is irreducible under $(\tilde{\pi}^{*1},\pi^{*2})$, by (\ref{eq 2.54}), we have
 $\psi^{*}=\kappa\tilde{\psi}$ in $S$. From (\ref{eq 2.46}) and (\ref{eq 2.50}) it follows that $\tilde{\rho} = \rho$. This proves (\ref{eq 2.47}).
 \end{proof}
In the next theorem we show that any mini-max selector of (\ref{eq 2.30}) is a saddle point equilibrium.
\begin{thm}\label{theo 3.2}
Suppose Assumptions \ref{assm 2.1}, \ref{assm 2.2}, and \ref{assm 2.3} hold. Then any mini-max selector $(\pi^{*1},\pi^{*2})\in \Pi^s_1\times \Pi^s_2$ of (\ref{eq 2.30}) is a saddle point equilibrium.
\end{thm}
\begin{proof}
Arguing as in Lemma \ref{lemm 2.6} and Theorem \ref{theo 3.1}, there exists $(\rho^{\pi^{*1},\pi^{*2}}, \psi^{\pi^{*1},\pi^{*2}}) \in \mathbb{R}_+\times L^\infty_V$ with $\psi^{\pi^{*1},\pi^{*2}}>0$ satisfying
\begin{align}
	\rho^{\pi^{*1},\pi^{*2}} \psi^{\pi^{*1},\pi^{*2}}(i) = \bigg[\sum_{j\in S}\psi^{\pi^{*1},\pi^{*2}}(j)q(j|i,\pi^{*1}(i), \pi^{*2}(i))+c(i,\pi^{*1}(i),\pi^{*2}(i))\psi^{\pi^{*1},\pi^{*2}}(i)\bigg]. \label{T3.2E3.3}
	\end{align} 
Furthermore  $\rho^{\pi^{*1},\pi^{*2}} = J(i,c,\pi^{*1},\pi^{*2})$ and for some finite set $\mathscr{B}\supset \hat{\mathscr{K}}$ (without loss of generality denoting by the same notation)
\begin{align}
\psi^{\pi^{*1},\pi^{*2}}(i) =  E^{\pi^{*1},\pi^{*2}}_i\bigg[e^{\int_{0}^{\tilde{\tau}(\mathscr{B})}(c(\xi_t,\pi^{*1}(\xi_t),\pi^{*2}(\xi_t))-\rho^{\pi^{*1},\pi^{*2}})dt}\psi^{\pi^{*1},\pi^{*2}}(\xi_{\tilde{\tau}(\mathscr{B})})\bigg]~\forall i\in \mathscr{B}^c.\label{T3.2E3.3A}
\end{align} Thus, from (\ref{eq 2.47}) it is clear that $\rho^{\pi^{*1},\pi^{*2}} \leq \rho$\,. Now, following  similar arguments as in Theorem \ref{theo 3.1} it is easy to see that $\rho^{\pi^{*1},\pi^{*2}} = \rho$\,. This implies that $J(i,c,\pi^1,\pi^{*2}) \leq \rho^{\pi^{*1},\pi^{*2}}$ for all $\pi^1\in\Pi^1$\,.
Next, from \cite{BP} it is clear that if we consider the minimization problem $\displaystyle \min_{\pi^2\in\Pi^2} J(i,c,\pi^{*1},\pi^2)$, then the optimal control exists in the space of stationary Markov strategies. Thus to complete the proof, it is enough to show that $J(i,c,\pi^{*1},\pi^{*2})\leq J(i,c,\pi^{*1},\pi^2)$ for any $\pi^2\in \Pi^s_2$\,. If not suppose that $J(i,c,\pi^{*1},\pi^{*2})> J(i,c,\pi^{*1},\pi^2)$ for some $\pi^2\in \Pi^s_2$\,. We know that for $\pi^2\in \Pi^s_2$, there exists $(\rho^{\pi^{*1},\pi^{2}}, \psi^{\pi^{*1},\pi^{2}}) \in \mathbb{R}_+\times L^\infty_V$ with $\psi^{\pi^{*1},\pi^{2}}>0$ satisfying
\begin{align}
	\rho^{\pi^{*1},\pi^{2}} \psi^{\pi^{*1},\pi^{2}}(i) = \bigg[\sum_{j\in S}\psi^{\pi^{*1},\pi^{2}}(j)q(j|i,\pi^{*1}(i), \pi^{2}(i))+c(i,\pi^{*1}(i),\pi^{2}(i))\psi^{\pi^{*1},\pi^{2}}(i)\bigg],\label{T3.2E3.3B}
	\end{align} also we have $\rho^{\pi^{*1},\pi^{2}} = J(i,c,\pi^{*1},\pi^{2})$ and for some finite set $\mathscr{B}$ ($\supset \hat{\mathscr{K}}$)
\begin{align}
\psi^{\pi^{*1},\pi^{2}}(i) =  E^{\pi^{*1},\pi^{2}}_i\bigg[e^{\int_{0}^{\tilde{\tau}(\mathscr{B})}(c(\xi_t,\pi^{*1}(\xi_t),\pi^{2}(\xi_t))- \rho^{\pi^{*1},\pi^{2}})dt}\psi^{\pi^{*1},\pi^{2}}(\xi_{\tilde{\tau}(\mathscr{B})})\bigg]~\forall i\in \mathscr{B}^c.\label{T3.3E3.3C}
\end{align} 	
From (\ref{eq 2.31}), we deduce that
\begin{align}
		\psi^{*}(i) \leq E^{\pi^{*1},\pi^2}_i\bigg[e^{\int_{0}^{\tilde{\tau}(\mathscr{B})}(c(\xi_t,\pi^{*1}(\xi_t),\pi^2(\xi_t))-\rho)dt}\psi^{*}(\xi_{\tilde{\tau}(\mathscr{B})})\bigg], ~\forall i\in \mathscr{B}^c. \label{T3.3E3.3D}
		\end{align} Now, as in Theorem \ref{theo 3.1}, using (\ref{T3.3E3.3C}) and (\ref{T3.3E3.3D}) one can deduce that $\psi^{\pi^{*1},\pi^{2}} = \eta \psi^{*}$ for some positive constant $\eta$\,. In view (\ref{eq 2.30}) and (\ref{T3.2E3.3B}), it follows that $\rho \leq \rho^{\pi^{*1},\pi^{2}}$, i.e., $J(i,c,\pi^{*1},\pi^{*2})\leq J(i,c,\pi^{*1},\pi^2)$, which is a contradiction. This completes the proof.
\end{proof}
Next we prove the converse of the above theorem.
\begin{thm}\label{theo 3.3}
Suppose Assumptions \ref{assm 2.1}, \ref{assm 2.2}, and \ref{assm 2.3} hold. If there exists a saddle point equilibrium $(\hat{\pi}^{*1},\hat{\pi}^{*2})\in \Pi^s_1\times \Pi^s_2$\,, i.e.,
\begin{align*}
J(i,c,\pi^{1},\hat{\pi}^{*2}) \leq  J(i,c,\hat{\pi}^{*1},\hat{\pi}^{*2}) \leq J(i,c,\hat{\pi}^{*1},\pi^2)\,,
\end{align*} for all $i\in S$\,, $\pi^1\in \Pi^1$ and $\pi^2\in \Pi^2$\,.
 Then $(\hat{\pi}^{*1},\hat{\pi}^{*2})$ is a mini-max selector of (\ref{eq 2.30}).
\end{thm}
\begin{proof}
From Theorem \ref{theo 3.2}, we deduce that
\begin{align*}
\rho = \inf_{\pi^2\in\Pi^2}\sup_{\pi^1\in\Pi^1}J(i,c,\pi^{1},\pi^{2}) & \leq \sup_{\pi^1\in\Pi^1}J(i,c,\pi^{1},\hat{\pi}^{*2}) \leq J(i,c,\hat{\pi}^{*1},\hat{\pi}^{*2}) \\
&\leq \inf_{\pi^2\in\Pi^2}J(i,c,\hat{\pi}^{*1},\pi^{2}) \leq \sup_{\pi^1\in\Pi^1} \inf_{\pi^2\in\Pi^2} J(i,c,\pi^{1},\pi^{2}) = \rho\,.
\end{align*}
This implies that $\rho = J(i,c,\hat{\pi}^{*1},\hat{\pi}^{*2})$\, and $\displaystyle \rho = \inf_{\pi^2\in\Pi^2}J(i,c,\hat{\pi}^{*1},\pi^{2})$\,. Arguing as in Lemma \ref{lemm 2.6} and using Theorem \ref{theo 3.1}, it follows that for $\hat{\pi}^{*1}\in\Pi_1^s$ there exists $(\rho_{\hat{\pi}^{*1}}, \psi_{\hat{\pi}^{*1}}^{*})\in \mathbb{R}_+\times L^\infty_V$ with $\psi_{\hat{\pi}^{*1}}^{*} > 0$ such that
\begin{align}
	\rho_{\hat{\pi}^{*1}}\psi_{\hat{\pi}^{*1}}^{*}(i)=\inf_{\nu\in P(B(i))}\bigg[\sum_{j\in S}\psi_{\hat{\pi}^{*1}}^{*}(j)q(j|i,\hat{\pi}^{*1}(i),\nu)+c(i,\hat{\pi}^{*1}(i),\nu)\psi_{\hat{\pi}^{*1}}^{*}(i)\bigg],\label{T3.3A}
	\end{align}
 and $\rho_{\hat{\pi}^{*1}} = \rho$ (since $\displaystyle \rho = \inf_{\pi^2\in\Pi^2}J(i,c,\hat{\pi}^{*1},\pi^{2})$). Let $({\pi}^{*1},{\pi}^{*2})$ be any mini-max selector of (\ref{eq 2.30}). Then form the above, we get
\begin{align}
	\rho_{\hat{\pi}^{*1}}\psi_{\hat{\pi}^{*1}}^{*}(i) \leq \bigg[\sum_{j\in S}\psi_{\hat{\pi}^{*1}}^{*}(j)q(j|i,\hat{\pi}^{*1}(i),\pi^{*2}(i))+c(i,\hat{\pi}^{*1}(i),\pi^{*2}(i))\psi_{\hat{\pi}^{*1}}^{*}(i)\bigg].\label{T3.3B}
	\end{align} Again arguing as in Lemma \ref{lemm 2.6}, for some $\mathscr{B} \supset \hat{\mathscr{K}}$ we have
\begin{align}
\psi_{\hat{\pi}^{*1}}^{*}(i) \leq  E^{\hat{\pi}^{*1},\pi^{*2}}_i\bigg[e^{\int_{0}^{\tilde{\tau}(\mathscr{B})}(c(\xi_t,\hat{\pi}^{*1}(\xi_t),\pi^{*2}(\xi_t))-\rho)dt}\psi_{\hat{\pi}^{*1}}^{*}(\xi_{\tilde{\tau}(\mathscr{B})})\bigg]~\forall i\in \mathscr{B}^c.\label{T3.3C}
\end{align}
Since, $({\pi}^{*1},{\pi}^{*2})$ is a mini-max selector of (\ref{eq 2.30}), we have
\begin{align*}
	\rho\psi^{*}(i) \geq \bigg[\sum_{j\in S}\psi^{*}(j)q(j|i,\hat{\pi}^{*1}(i),\pi^{*2}(i))+c(i,\hat{\pi}^{*1}(i),\pi^{*2}(i))\psi^{*}(i)\bigg].
	\end{align*}
	Thus, by applying Dynkin's formula and Fatou's lemma, we obtain
\begin{align}
\psi^{*}(i) \geq  E^{\hat{\pi}^{*1},\pi^{*2}}_i\bigg[e^{\int_{0}^{\tilde{\tau}(\mathscr{B})}(c(\xi_t,\hat{\pi}^{*1}(\xi_t),\pi^{*2}(\xi_t))-\rho)dt}\psi^{*}(\xi_{\tilde{\tau}(\mathscr{B})})\bigg]~\forall i\in \mathscr{B}^c.\label{T3.3D}
\end{align}	
Using (\ref{T3.3C}) and (\ref{T3.3D}), and following the arguments as in Theorem \ref{theo 3.1} one can show that $\psi^{*} = \hat{\eta}\psi_{\hat{\pi}^{*1}}^{*}$ for some positive constant $\hat{\eta}$. Therefore, combining (\ref{eq 2.30}) and (\ref{T3.3A}) it is easy to see that $\hat{\pi}^{*1}$ is an outer maximizing selector of (\ref{eq 2.30}). By  similar arguments we can show that $\hat{\pi}^{*2}$ is an outer minimizing selector of (\ref{eq 2.30}). This completes the proof.
\end{proof}
 \section{ Example}
In this section an illustrative example is presented. In our model the transition rate is unbounded, and the cost rate is nonnegative and unbounded.
\begin{example}
	Consider a controlled birth-death system in which the state variable denotes the total population at each time $t\geq 0$. In this system there are `natural' arrival and departure rates, say $\hat{\mu}$ and $\hat{\lambda}$, respectively. Here player 1 controls arrival parameters $\hat{h}_1$ and player 2 controls departure parameters $\hat{h}_2$. At any time $t$, when the state of the system is $i\in S:=\{0,1,\cdots\}$, player 1 takes an action $a$ from a given set $A(i)$ (which is a compact subset of some Polish space $A$). This action may increase $(\hat{h}_1(i,a)\geq 0)$ or decrease $(\hat{h}_1(i,a)\leq 0)$, the arrival rate and these actions result in a payoff denoted by $\hat{c}_1(i,a)$ per unit time. Similarly, if the state is $i\in \{1,2,\cdots\}$, player 2 takes an action $b$ from a set $B(i)$ (which is a compact subset of a Polish space $B$) to increase $(\hat{h}_2(i,b)\geq 0)$ or to decrease $(\hat{h}_2(i,b)\leq 0)$, the departure rate and these actions produce a payoff denoted by $\hat{c}_2(i,b)$ per unit time. Also, in addition, assume that player 1 `owns' the system and he/she gets a reward $\hat{p}\cdot i$ for each unit of time during which the system remains in the state $i\in S$, where $\hat{p}>0$ is a fixed reward fee per customer. We also, assume that when the state of the system reaches at state $i=0$, any number of arrivals may occur. When there is no customer in the system, (i.e., $i=0$), control of departure is unnecessary.
	
	We next formulate this model as a continuous-time Markov game. The corresponding transition rate $q(j|i,a,b)$ and reward rate $c(i,a,b)$ for player 1 are given as follows: for $(0,a,b)\in K$ ($K$ as in the game model (\ref{eq 2.1})). We take
	\begin{align}
	q(j|0,a,b)=\frac{\alpha}{(j+3)^4}~\text{ for all}~j\geq 1,~\text{such that}~\sum_{j\in S}q(j|0,a,b)=0,\label{eq 4.1}
	\end{align}
	where $\alpha>0$ is some constant so that $q(0|0,a,b)\leq -3$. Also for $(i,a,b)\in K$ with $i\geq 1$,
	\begin{align}
	q(j|i,a,b)= \left\{ \begin{array}{lll}\displaystyle{}&\hat{\lambda} (i+3)^2+\hat{h}_2(i,b),
	~\text{if}~j=i-1\nonumber\\
	&-\hat{\mu} i-\hat{\lambda} (i+3)^2-\hat{h}_1(i,a)-\hat{h}_2(i,b),~\text{if}~ j=i\\
	& \hat{\mu} i+\hat{h}_1(i,a),~\text{if}~j=i+1\\
	&0,~\text{otherwise}\displaystyle{}.
	\end{array}\right.
	\end{align}
	\begin{align}
	c(i,a,b):=\hat{p}\cdot i-\hat{c}_1(i,a)+\hat{c}_2(i,b)~\text{ for }~(i,a,b)\in K.\label{eq 4.2}
	\end{align}
	We now explore conditions under which there exists a pair of optimal strategies. To do so, we make the following assumptions.
	\begin{enumerate}
		\item [(I)] Let $\hat{\lambda}\geq \max\{\hat{\mu},2\}$, $\hat{\mu}i+\hat{h}_1(i,a)>0$, and $\hat{\lambda} (i+3)^2+\hat{h}_2(i,b)>0$ for all $(i,a,b)\in K$ with $i\geq 1$; and  assume that $\hat{h}_1(0,a)>0$ and $\hat{h}_2(0,b)=0$ for all $(a,b)\in A(i)\times B(i).$
		
		\item [(II)] The functions $\hat{h}_1(\cdot,\cdot):S\times A\rightarrow [-\hat{\mu},\hat{\mu}]$, $\hat{h}_2(\cdot,\cdot):S\times B\rightarrow [-\hat{\lambda},\hat{\lambda}]$, $\hat{c}_1(i,a)$, and $\hat{c}_2(i,b)$ are continuous with their respective variables for each fixed $i\in S$.
		Also, assume that $\displaystyle \min_{(a,b)\in A(i)\times B(i)}[\hat{c}_1(i,a)-\hat{c}_2(i,b)]$ is norm-like function and $\hat{p}\cdot i-\hat{c}_1(i,a)+\hat{c}_2(i,b)\geq 0$ for $(i,a,b)\in K$. Here we take $\hat{p}<1$.
	\end{enumerate}
\end{example}

\begin{proposition}\label{Prop 6.1}
	Under conditions (I)-(II), the above controlled birth-death system satisfies the Assumptions \ref{assm 2.1}, \ref{assm 2.2}, and \ref{assm 2.3}. Hence by Theorem {\ref{theo 3.2}}, there exists a pair of optimal strategies.
\end{proposition}
\begin{proof}
	Consider the Lyapunov function given by 
	\begin{align*}
	V(i):=(i+3)^2~\text{ for }~i\in S.
	\end{align*}
	We have $V(i)\geq 1$ for all $i\in S$. Now for each $i\geq 1$, and $(a,b)\in A(i)\times B(i)$, we have
	\interdisplaylinepenalty=0
	\begin{align}
	\sum_{j\in S}q(j|i,a,b)V(j)&=q(i-1|i,a,b)V(i-1)+V(i)q(i|i,a,b)+V(i+1)q(i+1|i,a,b)\nonumber\\
	&=(i+2)^2[\hat{\lambda}(i+3)^2+\hat{h}_2(i,b)]-(i+3)^2[\hat{\mu}i+\hat{\lambda} (i+3)^2+\hat{h}_1(i,a)+\hat{h}_2(i,b)]\nonumber\\
	&\quad+(i+4)^2[\hat{\mu}i+\hat{h}_1(i,a)]\nonumber\\
	&=-[\hat{\lambda}(i+3)^2+\hat{h}_2(i,b)](2i+5)+(\hat{\mu}i+\hat{h}_1(i,a))(2i+7)\nonumber\\
	&=-\hat{\lambda}(i+3)^2(i+3+i+2)-\hat{h}_2(i,b)(2i+5)+(\hat{\mu}i+\hat{h}_1(i,a))(2i+7)\nonumber\\
	&=-\hat{\lambda}(i+3)^2(i+3)-\hat{\lambda}(i+3)^2(i+2)-\hat{h}_2(i,b)(2i+5)+(\hat{\mu}i+\hat{h}_1(i,a))(2i+7)\nonumber\\
	&\leq -\hat{\lambda}(i+3)(i+3)^2-\hat{\lambda} (i+3)^2(i+2)+\hat{\lambda}(2i+5)+\hat{\lambda}i(2i+7)+\hat{\lambda}(2i+7)\nonumber\\
	&~(\text{ since} -h_2(i,b)\leq \hat{\lambda},~ \hat{\mu}\leq \hat{\lambda},~ h_1(i,a)\leq \hat{\mu}\leq \hat{\lambda}, \text{ by conditions (I) and (II)})\nonumber\\
	&=-\frac{\hat{\lambda}}{2}(i+3)V(i)+\biggl\{-\frac{\hat{\lambda}}{2}(i+3)(i+3)^2-\hat{\lambda} (i+3)^2(i+2)+\hat{\lambda}(2i+5)\nonumber\\
	&+\hat{\lambda}i(2i+7)+\hat{\lambda}(2i+7)\biggr\}\nonumber\\
	&\leq -\frac{\hat{\lambda}}{2}(i+3)V(i)~(\text{ since the term within the second bracket is negative})\nonumber\\
	&\leq -(i+3)V(i)~(\text{ by condition (I), since } \hat{\lambda}\geq 2)\nonumber\\
	&=-\hat{\ell}(i)V(i)\label{eq 4.3}
	\end{align}
	where $\hat{\ell}(i)=i+3$.
	For $i=0$,
	\begin{align}
	\sum_{j\in S}q(j|i,a,b)V(j)&=9q(0|0,a,b)+\sum_{j\geq 1}q(j|0,a,b)(j+3)^2\nonumber\\
	&\leq CI_{\hat{\mathscr{K}}}(0)-\hat{\ell}(0)V(0),\label{eq 4.4}
	\end{align}
	where $\hat{\mathscr{K}}=\{0\}$ and $C=\frac{\alpha\pi^2}{6}$.
	Now
	\begin{align}
	\sup_{(a,b)\in A(i)\times B(i)}q(i,a,b)&\leq \sup_{(a,b)\in A(i)\times B(i)}\biggl\{\hat{\mu} i+\hat{\lambda} (i+3)^2+|\hat{h}_1(i,a)|+|\hat{h}_2(i,b)|\biggr\}\nonumber\\
	&\leq \hat{\lambda} i+2\hat{\lambda}+\hat{\lambda} (i+3)^2\nonumber\\
	&\leq 2(i+3)^2\hat{\lambda}
	\leq b_2 V(i)~\forall~i\geq 1,\label{eq 4.5}
	\end{align}
	where $b_2=\max\{2\hat{\lambda},\displaystyle \sum_{j\geq 1}\frac{\alpha}{{(j+3)}^4}\}$.
	From (\ref{eq 4.3}) and (\ref{eq 4.4}), for all $i\in S$, we get
	\begin{align}
	\sum_{j\in S}q(j|i,a,b)V(j)\leq b_0V(i)+b_1,\label{eq 4.6}
	\end{align}
	where $b_1=C$ and $b_0=1$.
	Now
	\begin{align}
	\hat{\ell}(i)-\max_{(a,b)\in A(i)\times B(i)}c(i,a,b)=3+(1-\hat{p})i+\min_{(a,b)\in A(i)\times B(i)}[\hat{c}_1(i,a)-\hat{c}_2(i,b)].\label{eq 4.7}
	\end{align}
	
	From (\ref{eq 4.5}) and (\ref{eq 4.6}), Assumption \ref{assm 2.1} is verified.
	By the condition (II), equations (\ref{eq 4.3}), (\ref{eq 4.4}), (\ref{eq 4.7}), it is easy to see that Assumption \ref{assm 2.2} is verified.
	By (\ref{eq 4.1}), (\ref{eq 4.2}), the condition (II), and the definition of $q$ as defined above,  Assumption \ref{assm 2.3} (i) is verified. By (\ref{eq 4.3}) and (\ref{eq 4.4}), Assumption \ref{assm 2.3} (ii) is verified.
	Hence by Theorem {\ref{theo 3.2}}, it follows that there exists an optimal pair of stationary strategies for this controlled Birth-Death process.
\end{proof}

  		\bibliographystyle{elsarticle-num}

	 \nocite{*}
	\bibliographystyle{plain}
	\end{document}